\documentclass[a4]{amsart}

\input xypic
\usepackage{amssymb}

\usepackage{color}

\newtheorem{theorem}{Theorem}[section]
\newtheorem{proposition}[theorem]{Proposition}
\newtheorem{lemma}[theorem]{Lemma}

\theoremstyle{definition}

\newtheorem{remark}[theorem]{Remark}

\newcommand{\gk}{\ensuremath{\mathcal{G}_{k}}}
\newcommand{\gkprime}{\ensuremath{\mathcal{G}_{k^{\prime}}}}
\newcommand{\map}{\ensuremath{\mbox{Map}}}
\newcommand{\mapstar}{\ensuremath{\mbox{Map}^{\ast}}} 
 
\newcommand{\cpthree}{\ensuremath{\mathbb{C}P^{3}}}

\newcommand{\cohlgy}[1]{\ensuremath{H^{*}(#1)}} 
\newcommand{\rcohlgy}[1]{\ensuremath{\widetilde{H}^{*}(#1)}}

\newcounter{bean}
\newenvironment{letterlist}{\begin{list}{\rm ({\alph{bean}})}
      {\usecounter{bean}\setlength{\rightmargin}{\leftmargin}}}
      {\end{list}}

\newcommand{\namedright}[3]{\ensuremath{#1\stackrel{#2}
 {\longrightarrow}#3}}
\newcommand{\nameddright}[5]{\ensuremath{#1\stackrel{#2}
 {\longrightarrow}#3\stackrel{#4}{\longrightarrow}#5}}
\newcommand{\namedddright}[7]{\ensuremath{#1\stackrel{#2}
 {\longrightarrow}#3\stackrel{#4}{\longrightarrow}#5
  \stackrel{#6}{\longrightarrow}#7}}

\newcommand{\larrow}{\relbar\!\!\relbar\!\!\rightarrow}
\newcommand{\llarrow}{\relbar\!\!\relbar\!\!\larrow}
\newcommand{\lllarrow}{\relbar\!\!\relbar\!\!\llarrow} 
\newcommand{\llllarrow}{\relbar\!\!\relbar\!\!\lllarrow} 

\newcommand{\lnamedright}[3]{\ensuremath{#1\stackrel{#2}
 {\larrow}#3}}
\newcommand{\lnameddright}[5]{\ensuremath{#1\stackrel{#2}
 {\larrow}#3\stackrel{#4}{\larrow}#5}}

\newcommand{\llnamedright}[3]{\ensuremath{#1\stackrel{#2}
 {\llarrow}#3}}
\newcommand{\llnameddright}[5]{\ensuremath{#1\stackrel{#2}
 {\llarrow}#3\stackrel{#4}{\llarrow}#5}}

\newcommand{\lllnameddright}[5]{\ensuremath{#1\stackrel{#2}
 {\lllarrow}#3\stackrel{#4}{\lllarrow}#5}}

\newcommand{\llllnameddright}[5]{\ensuremath{#1\stackrel{#2}
 {\llllarrow}#3\stackrel{#4}{\llllarrow}#5}}

\newcommand{\qqed}{\hfill\Box}

\begin{document}


\title{The homotopy types of $SU(4)$-gauge groups} 
\author{Tyrone Cutler}
\address{Fakult{\"a}t f{\"u}r Mathematik, Universit{\"a}t Bielefeld, 33615 Bielefeld, Germany}
\email{tcutler@math.uni-bielefeld.de} 
\author{Stephen Theriault}
\address{Mathematical Sciences, University of Southampton, Southampton 
               SO17 1BJ, United Kingdom}
\email{S.D.Theriault@soton.ac.uk} 



\begin{abstract} 
Let \gk\ be the gauge group of the principal $SU(4)$-bundle over $S^{4}$ 
with second Chern class $k$ and let $p$ be a prime. We show that there 
is a rational or $p$-local homotopy equivalence $\Omega\gk\simeq\Omega\gkprime$ 
if and only if $(60,k)=(60,k^{\prime})$. 
\end{abstract}

\maketitle

\section{Introduction}
\label{sec:intro}

Let $G$ be a simply-connected, simple compact Lie group. 
Then principal $G$-bundles over $S^{4}$ are classified by the value 
of a degree $4$ characteristic class. For instance, if $G=SU(n)$, then this is the second Chern class. 
Fixing a generator we obtain an isomorphism $H^4(S^4)\cong\mathbb{Z}$,
and this class can take any integer value. Let 
\(\namedright{P_{k}}{}{S^{4}}\) 
represent the equivalence class of principal $G$-bundle
corresponding under the above scheme to the integer $k\in\mathbb{Z}$.
Let~\gk\ be the \emph{gauge group} 
of this principal $G$-bundle, which is the group of $G$-equivariant 
automorphisms of $P_{k}$ over $S^{4}$.
 
Crabb and Sutherland~\cite{CS} showed that, while there are countably 
many inequivalent principal $G$-bundles, the gauge groups 
$\{\gk\}_{k\in\mathbb{Z}}$ have only finitely many distinct homotopy 
types. There has been a great deal of interest recently in determining 
the precise number of possible homotopy types. The following 
classifications are known. For two integers $a,b$, let $(a,b)$ be 
their greatest common divisor. If $G=SU(2)$ then 
$\gk\simeq\gkprime$ if and only if $(12,k)=(12,k^{\prime})$~\cite{K}; 
if $G=SU(3)$ then $\gk\simeq\gkprime$ if and only if 
$(24,k)=(24,k^{\prime})$~\cite{HK}; if $G=SU(5)$ then 
$\gk\simeq\gkprime$ when localized at any prime $p$ or 
rationally if and only if $(120,k)=(120,k^{\prime})$~\cite{Th2}; and if 
$G=Sp(2)$ then $\gk\simeq\gkprime$ when localized at any prime $p$ 
or rationally if and only if $(40,k)=(40,k^{\prime})$~\cite{Th1}. Partial 
classifications that are potentially off by a factor of $2$ have been 
worked out for $G_{2}$~\cite{KTT} and $Sp(3)$~\cite{Cu}. 

The $SU(4)$ case is noticeably absent. The $SU(5)$ case was easier  
since elementary bounds on the number of homotopy types matched at 
the prime $2$ but did not at the prime $3$, and it is typically easier to 
work out $3$-primary problems in low dimension than $2$-primary 
problems. In the $SU(4)$ case the elementary bounds do not match at $2$, and 
the purpose of this paper is to resolve the difference, at least after looping.  

\begin{theorem} 
   \label{types} 
   For $G=SU(4)$, there is a homotopy equivalence 
   $\Omega\gk\simeq\Omega\gkprime$ when localized at any prime $p$ or 
   rationally if and only if $(60,k)=(60,k^{\prime})$.  
\end{theorem}  

Two novel features arise in the methods used, as compared to the other 
known classifications. One is the use of Miller's stable splittings of Stiefel 
manifolds in order to gain some control over unstable splittings, and the 
other is showing that a certain ambiguity which prevents a clear 
classification statement for $\gk$ vanishes after looping. It would be 
interesting to know if these ideas give access to classifications for 
$SU(n)$-gauge groups for $n\geq 6$. 

One motivation for studying $SU(4)$-gauge groups is their connection 
to physics, in particular, to $SU(n)$-extensions of the standard model. 
For instance, the group $SU(4)$ is gauged in the Pati-Salam model~\cite{PS} 
and the flavour symmetry it represents there plays a role in several other 
grand unified theories~\cite{BH}. The progression of results from $SU(2)$ 
to $SU(5)$ and possibly beyond would be of interest to physicists studying 
the $SU(n)$-gauge groups in t'Hooft's large $n$ expansion~\cite{tH}. 

The authors would like to thank Michael Crabb for helpful discussions regarding 
stable splittings of Stiefel manifolds, and the referee for a very careful reading of the paper 
and the many suggestions for improvement.

\section{Determining homotopy types of gauge groups} 
\label{sec:prelim} 

We begin by describing a context in which homotopy theory can be applied 
to study gauge groups. This works for any simply-connected, simple 
compact Lie group $G$ and so is stated that way. Let $BG$ and $B\gk$ 
be the classifying spaces of $G$ and $\gk$ respectively. 
Let $\map(S^{4},BG)$ and $\mapstar(S^{4},BG)$ respectively 
be the spaces of freely continuous and pointed continuous maps 
between~$S^{4}$ and $BG$. The components of each space are 
in one-to-one correspondence with the integers, where the 
integer is determined by the degree of a map 
\(\namedright{S^{4}}{}{BG}\). 
By~\cite[Proposition~2.4]{AB} or~\cite[Theorem 5.2]{G}, there is a homotopy equivalence 
$B\gk\simeq\map_{k}(S^{4},BG)$ between $B\gk$ and the component 
of $\mbox{Map}(S^{4},BG)$ consisting of maps of degree~$k$. 
Evaluating a map at the basepoint of~$S^{4}$, we obtain a map 
\(ev\colon\namedright{B\gk}{}{BG}\) 
whose fibre is homotopy equivalent to $\mapstar_{k}(S^{4},BG)$. 
It is well known that each component of $\mapstar(S^{4},BG)$ is 
homotopy equivalent to $\Omega^{3}_{0} G$, the component of 
$\Omega^{3} G$ containing the basepoint. Putting all this together, 
for each $k\in\mathbb{Z}$, there is a homotopy fibration sequence 
\begin{equation} 
  \label{evfib} 
  \namedddright{G}{\partial_{k}}{\Omega^{3}_{0} G}{}{B\gk}{ev}{BG} 
\end{equation} 
where $\partial_{k}$ is the fibration connecting map. 

The order of $\partial_{k}$ plays a crucial role. By~\cite[Theorem 2.6]{L}, the triple adjoint 
\(\namedright{S^{3}\wedge G}{}{G}\) 
of $\partial_{k}$ is homotopic to the Samelson product 
$\langle k\cdot i,1\rangle$, where $i$ is the inclusion of 
$S^{3}$ into $G$ and $1$ is the identity map on~$G$. This implies 
two things. First, the order of $\partial_{k}$ is finite. For, rationally, 
$G$ is homotopy equivalent to a product of Eilenberg-MacLane spaces,
and moreover this equivalence can be induced by an $H$-map. 
Indeed, according to the Hopf-Borel Theorem~\cite[p.16]{Ka}, $H^*(G;\mathbb{Q})$ is a primitively generated Hopf algebra. 
Any choice of primitive generators will yield a map with the required properties.
Since Eilenberg-MacLane spaces are homotopy commutative, 
any Samelson product into such a space is null homotopic. Thus, rationally, the 
adjoint of $\partial_{k}$ is null homotopic, implying that 
the same is true for $\partial_{k}$ and therefore that the order of 
$\partial_{k}$ is finite. Second, the linearity of the Samelson product  
implies that $\langle k\cdot i,1\rangle\simeq k\circ\langle i,1\rangle$, 
so taking adjoints we obtain $\partial_{k}\simeq k\circ\partial_{1}$. 
Thus the order of~$\partial_{k}$ is determined by the order of $\partial_{1}$. 
When $G=SU(n)$, lower bounds exist on the order of $\partial_{1}$ 
and on the number of homotopy types of \gk.

\begin{lemma} 
   \label{HKlemma} 
   Let $G=SU(n)$. If $n>2$ then the following hold: 
   \begin{letterlist} 
      \item the order of $\partial_{1}$ is a multiple of $n(n^{2}-1)$; 
      \item if $\gk\simeq\gkprime$ then 
                    $(n(n^{2}-1),k)=(n(n^{2}-1),k^{\prime})$. 
   \end{letterlist} 
\end{lemma} 

\begin{proof} 
Consider the Samelson product 
\(\namedright{S^{3}\wedge SU(n)}{\langle i,1\rangle}{SU(n)}\) 
where $i$ is the inclusion of the bottom cell and $1$ is the identity map. 
Bott~\cite[Theorem 1]{B} showed that if 
\(c'\colon\namedright{S^{2n-3}}{}{SU(n)}\) 
and 
\(c\colon\namedright{S^{2n-1}}{}{SU(n)}\) 
represent generators of $\pi_{2n-3}(SU(n))\cong\mathbb{Z}$ and 
$\pi_{2n-1}(SU(n))\cong\mathbb{Z}$, respectively, then 
the Samelson product $\langle i,c'\rangle=\langle i,1\rangle\circ (1\wedge c')$ 
has order $n(n-1)$ while the Samelson product $\langle i,c\rangle=\langle i,1\rangle\circ(1\wedge c)$ 
has order $n(n+1)/2$ if $n$ is odd and $n(n+1)$ if $n$ is even. Thus if $n$ is even the order 
of $\langle i,1\rangle$ is at least $n(n^{2}-1)$ (that is, the order of $\langle i,1\rangle$ 
is a multiple of $n(n^{2}-1)$). As $\partial_{1}$ is the adjoint of $\langle i,1\rangle$, it has the same 
order, and hence the order of $\partial_{1}$ is a multiple of $n(n^{2}-1)$ if~$n$ is even. 

If $n$ is odd then this homotopy group calculation differs from the statement of part~(a) 
by a factor of~$\frac{1}{2}$. On the other hand, Hamanaka and Kono~\cite[calculation 
preceding Lemma 2.5]{HK} showed that there is a map 
\(d\colon\namedright{\Sigma^{2n-5}\mathbb{C}P^{2}}{}{SU(n)}\) 
with the property that the Samelson product $\langle i,d\rangle=\langle i,1\rangle\circ (1\wedge d)$ 
has order $n(n^{2}-1)$. Thus, as before, $\partial_{1}$ has order a multiple of $n(n^{2}-1)$. 
This proves part~(a). 

Part~(b) for $n$ even case is sketched by Sutherland~\cite[Proposition 4.2]{S} while the $n$ odd 
case was proved by Hamanaka and Kono~\cite[Theorem 1.2]{HK}. 
\end{proof} 

As we aim for statements about looped gauge groups, we need a looped version 
of Lemma~\ref{HKlemma}. 

\begin{lemma} 
   \label{loopHKlemma} 
   Let $G=SU(n)$. If $n>2$ then the following hold: 
   \begin{letterlist} 
      \item the order of $\Omega\partial_{1}$ is divisible by $n(n^{2}-1)$; 
      \item if $\Omega\gk\simeq\Omega\gkprime$ then 
                    $(n(n^{2}-1),k)=(n(n^{2}-1),k^{\prime})$. 
   \end{letterlist} 
\end{lemma} 

\begin{proof} 
The calculations described in the proof of Lemma~\ref{HKlemma}~(a) involved 
maps $c'$, $c$ and $d$, all of which were suspensions. Their adjoints therefore 
have the same order, so part~(a) follows. 

Suppose that $\Omega\gk\simeq\Omega\gkprime$. Then $[X,\Omega\gk]\cong [X,\Omega\gkprime]$ 
for any $CW$-complex $X$, implying that $[\Sigma^{2} X,B\gk]\cong [\Sigma^{2}X, B\gkprime]$.  
Therefore Sutherland's homotopy group calculations in~\cite[Example 4.1]{S} for $B\gk$ 
that led to the $n$ even case of Lemma~\ref{HKlemma}~(b) equally imply in our case that 
$(n(n^{2}-1),k)=(n(n^{2}-1),k^{\prime})$. Also, Hamanaka and Kono's calculation 
of $[\Sigma^{2n-5}\mathbb{C}P^{2}, B\gk]$ in~\cite[Lemma 2.5; see also page 150]{HK}, which 
they used to prove the $n$ odd case of Lemma~\ref{HKlemma}~(b), equally implies in 
our case that $(n(n^{2}-1),k)=(n(n^{2}-1),k^{\prime})$. 
\end{proof} 

In particular, if $G=SU(4)$ then $60$ divides the order of $\Omega\partial_{1}$ 
and a homotopy equivalence $\Omega\gk\simeq\Omega\gkprime$ implies that 
$(60,k)=(60,k^{\prime})$. In Section~\ref{sec:proof} we will find an upper bound 
on the order of $\Omega\partial_{1}$ that matches the lower bound. 

\begin{theorem} 
   \label{looppartialorder} 
   The map 
   \(\namedright{\Omega SU(4)}{\Omega\partial_{1}}{\Omega^{4}_{0} SU(4)}\) 
   has order~$60$. 
\end{theorem}

Granting Theorem~\ref{looppartialorder} for now, we can prove 
Theorem~\ref{types} by using the following general result from~\cite[Lemma 3.1]{Th1}. 
If $Y$ is an $H$-group (a homotopy associative $H$-space with a homotopy inverse), let 
\(k\colon\namedright{Y}{}{Y}\)
be the $k^{th}$-power map.  

\begin{lemma}
   \label{ptypecount}
   Let $X$ be a space and $Y$ be an $H$-group. Suppose there is a map
   \(\namedright{X}{f}{Y}\)
   of order~$m$, where $m$ is finite. Let $F_{k}$ be the homotopy fibre 
   of $k\circ f$. If $(m,k)=(m,k^{\prime})$ then $F_{k}$ and $F_{k^{\prime}}$ 
   are homotopy equivalent when localized rationally or at any prime.~$\qqed$ 
\end{lemma} 

\noindent 
\begin{proof}[Proof of Theorem~\ref{types}] 
By Theorem~\ref{looppartialorder}, the map 
\(\namedright{\Omega SU(4)}{\Omega\partial_{1}}{\Omega^{4}_{0} SU(4)}\) 
has order~$60$. It follows from Lemma~\ref{ptypecount}, therefore, that if 
$(60,k)=(60,k^{\prime})$, then $\Omega\gk\simeq\Omega\gkprime$ when 
localized at any prime $p$ or rationally. On the other hand, 
by Lemma~\ref{loopHKlemma}, if $\Omega\gk\simeq\Omega\gkprime$ then 
$(60,k)=(60,k^{\prime})$. Thus there is a homotopy equivalence 
$\Omega\gk\simeq\Omega\gkprime$ at each prime $p$ and rationally if and 
only if $(60,k)=(60,k^{\prime})$. 
\end{proof} 

It remains to prove Theorem~\ref{looppartialorder}. In fact, the odd primary 
components of the order of $\partial_{1}$ (and hence of $\Omega\partial_{1}$ 
by Lemma~\ref{loopHKlemma}) are obtained  as special cases of a more general 
result in~\cite[Theorem~1.1(c)]{Th3}. 

\begin{lemma} 
   \label{oddbounds} 
   Localized at $p=3$, $\partial_{1}$ has order $3$; localized at $p=5$, 
   $\partial_{1}$ has order $5$; and localized at~$p>5$, $\partial_{1}$ has 
   order~$1$.~$\qqed$ 
\end{lemma}  

Thus to prove Theorem~\ref{looppartialorder} we are reduced to proving 
the following. 

\begin{theorem} 
   \label{looppartialorder2} 
    Localized at $2$ the map 
   \(\namedright{\Omega SU(4)}{\Omega\partial_{1}}{\Omega^{4}_{0} SU(4)}\) 
   has order~$4$. 
\end{theorem} 

\section{An initial upper bound on the $2$-primary order of $\partial_{1}$} 
\label{sec:initialbound} 

Throughout this section all spaces and maps are localized at $2$. 
As mentioned in Section~\ref{sec:prelim}, the adjoint of $\partial_{1}$ is the
Samelson product \(\namedright{S^{3}\wedge SU(n)}{\langle i,1\rangle}{SU(n)}\), which is determined by the commutator in $SU(n)$. 
Since $SU(n)$ contains a 
copy of $SU(n-2)$ which commutes with the image of $i$, the Samelson product factors through
$S^{3}\wedge(SU(n)/SU(n-2))$ as in~\cite[page 251]{B}. Thus taking adjoints yields a homotopy commutative square
\[\diagram 
       SU(n)\rto^-{\partial_{1}}\dto^{q} & \Omega^{3}_{0} SU(n)\ddouble \\ 
       SU(n)/SU(n-2)\rto^-{f} & \Omega^{3}_{0} SU(n) 
  \enddiagram\] 
for some map $f$, where $q$ is the standard quotient map. 
In our case, by~\cite[Theorem 1.18]{JW} there is a homotopy equivalence 
$SU(4)/SU(2)\simeq S^{5}\times S^{7}$. Thus there is a homotopy 
commutative square 
\begin{equation} 
  \label{su4su2} 
  \diagram 
     SU(4)\rto^-{\partial_{1}}\dto^{q} & \Omega^{3}_{0} SU(4)\ddouble \\ 
     S^{5}\times S^{7}\rto^-{f} & \Omega^{3}_{0} SU(4). 
  \enddiagram 
\end{equation} 
Taking the triple adjoint of $f$, we obtain a map 
\[f'\colon\nameddright{S^{8}\vee S^{10}\vee S^{15}}{\simeq} 
       {\Sigma^{3}(S^{5}\times S^{7})}{}{SU(4)}.\] 
Mimura and Toda~\cite[Theorem 6.1]{MT} calculated the homotopy groups of $SU(4)$ 
through a range. The $2$-primary components of $\pi_{8}(SU(4))$, 
$\pi_{10}(SU(4))$ and $\pi_{15}(SU(4))$ are 
$\mathbb{Z}/8\mathbb{Z}$, $\mathbb{Z}/8\mathbb{Z}\oplus\mathbb{Z}/2\mathbb{Z}$ 
and $\mathbb{Z}/8\mathbb{Z}\oplus\mathbb{Z}/2\mathbb{Z}$, respectively. 
Consequently, the order of $f'$ is bounded above by $8$. The order of~$f$ 
is therefore also bounded above by $8$. The homotopy commutativity 
of~(\ref{su4su2}) then implies the following. 

\begin{lemma} 
   \label{initialbound} 
   Localized at $2$, the order of the map 
   \(\namedright{SU(4)}{\partial_{1}}{\Omega^{3}_{0} SU(4)}\) 
   is bounded above by $8$.~$\qqed$ 
\end{lemma} 

Ideally it should be possible to reduce this upper bound by a factor of two. 
The remainder of the paper aims to show that this can be done after looping.

\section{Stable splittings of Stiefel manifolds} 
\label{sec:stableStiefel} 

In obtaining a $2$-primary bound on the order of $\Omega\partial_1$ we will make 
use of the quotient map 
\(\namedright{SU(4)}{q}{SU(4)/SU(2)}\). 
This will be examined in detail after three suspensions, corresponding to the three loops 
on $\Omega^{3}_{0} SU(4)$. It helps to first have information about its stable behaviour. 
Miller~\cite[Theorem~C]{M} gave stable splittings of Stiefel manifolds, which would apply 
to both $SU(4)=U(4)/U(1)$ and $SU(4)/SU(2)\cong U(4)/U(2)$. These splittings  
satisfy certain naturality properties but are unfortunately not natural with respect to $q$. In 
Section~\ref{sec:stable} we will prove an \emph{ad hoc} stable splitting of $SU(4)$ 
that is compatible with the quotient map $q$ to $SU(4)/SU(2)\simeq S^{5}\times S^{7}$. 
This section is devoted to reviewing information about the stable splittings of 
Stiefel manifolds that will be used in the next section. 

We follow the exposition by Crabb~\cite{Cr}. To set notation, let $U$ and $V$ be non-zero 
finite-dimensional complex inner product space and let $n$ be the dimension of $V$ 
over $\mathbb{C}$. Let 
\begin{itemize} 
   \item $U(V)$ be the unitary group of $V$; 
   \item $\mathfrak{u}(V)$ be the Lie algebra of skew-Hermitian endomorphisms of $V$; 
   \item $G_{k}(V)$ be the Grassmann manifold of $k$-dimensional subspaces of $V$; 
   \item $\zeta_{k}$ be the canonical $k$-dimensional sub-bundle of the trivial bundle 
            over $G_{k}(V)$;  
   \item $\zeta_{k}^{\ast}\otimes W=\mbox{Hom}(\zeta_{k},W)$; 
   \item $G_{k}(V)^{\mathfrak{u}(\zeta_{k})\oplus(\zeta_{k}^{\ast}\otimes W)}$ be the Thom space of 
            $\mathfrak{u}(\zeta_{k})\oplus(\zeta_{k}^{\ast}\otimes W)$; 
   \item $U(V;W)$ be the space of isometric linear embeddings 
            \(\namedright{V}{}{V\oplus W}\), 
            which naturally identifies with $U(V\oplus W)/U(W)$. 
\end{itemize} 
A subscript $+$ indicates a disjoint basepoint. 

For $0\leq k\leq n$, Crabb~\cite[Section 1]{Cr} constructs $U(V)\times U(W)$-equivariant maps 
\[\sigma_{k}(V,W)\colon\namedright{G_{k}(V)^{\mathfrak{u}(\zeta_{k})\oplus(\zeta_{k}^{\ast}\otimes W)}} 
       {}{U(V;W)_{+}}\] 
that have left homotopy inverses stably. Collectively, they give a stable splitting 
\[\sigma(V,W)\colon\namedright 
     {\bigvee_{k=0}^{n} G_{k}(V)^{\mathfrak{u}(\zeta_{k})\oplus(\zeta_{k}^{\ast}\otimes W)}} 
     {}{U(V;W)\cong U(V\oplus W)/U(W)_{+}}.\]  
In particular, if $V=\mathbb{C}^{n}$ and $W=\mathbb{C}^{t}$ then there is a stable splitting 
\begin{equation} 
   \label{Crabbsplitting} 
   \bigvee_{k=0}^{n} G_{k}(\mathbb{C}^{n})^{\mathfrak{u}(\zeta_{k})\oplus(\zeta_{k}^{\ast}\otimes\mathbb{C}^{t})} 
        \simeq U(n+t)/U(t)_{+}. 
\end{equation} 

The authors are indebted to Michael Crabb for pointing out the following two naturality 
properties of the splitting maps $\sigma_{k}(V,W)$. 

\begin{lemma} 
   \label{Crabb1} 
   There is a commutative diagram of stable maps 
   \[\diagram 
         G_{k}(V)^{\mathfrak{u}(\zeta_{k})\oplus(\zeta_{k}^{\ast}\otimes W)}\rrto^-{\sigma_{k}(V,W)}\dto 
             &  & U(V;W)_{+}\dto \\ 
         G_{k}(V\oplus F)^{\mathfrak{u}(\zeta_{k})\oplus(\zeta_{k}^{\ast}\otimes W)}\rrto^-{\sigma_{k}(V\oplus F,W)}  
             &  & U(V\oplus F;W)_{+} 
      \enddiagram\] 
    in which the vertical maps are inclusions induced, on the left, by the inclusion 
    \(\namedright{V}{}{V\oplus W}\) 
    and, on the right, by taking the direct sum with the identity map 
    \(\namedright{F}{}{F}\).~$\qqed$  
\end{lemma} 

\begin{lemma} 
   \label{Crabb2} 
   There is a commutative diagram of $U(V)\times U(W)$-equivariant stable maps 
   \[\diagram 
         G_{k}(V)^{\mathfrak{u}(\zeta_{k})\oplus(\zeta_{k}^{\ast}\otimes W)}\rrto^-{\sigma_{k}(V,W)}\dto 
             & &  U(V;W)_{+}\dto \\ 
         G_{k}(V)^{\mathfrak{u}(\zeta_{k})\oplus(\zeta_{k}^{\ast}\otimes (E\oplus W))}\rrto^-{\sigma_{k}(V,E\oplus W)}  
             &  & U(V;E\oplus W)_{+} 
      \enddiagram\] 
    in which the vertical maps are induced by the inclusion 
    \(\namedright{W}{}{E\oplus W}\).~$\qqed$  
\end{lemma} 

Taking $V=\mathbb{C}^{n}$, $W=\mathbb{C}^{t}$ and $F=\mathbb{C}^{m}$, Lemma~\ref{Crabb1} 
implies that the stable splitting in~(\ref{Crabbsplitting}) is  compatible with the inclusion 
\[\namedright{U(n+t)/U(t)}{}{U(m+n+t)/U(t)}.\] 
Thus as $m\rightarrow\infty$ the stable splitting of $U(n+t)/U(t)$ is compatible with a stable 
splitting of $U(\infty)/U(t)$. Taking $E=\mathbb{C}^{r}$ as well, Lemma~\ref{Crabb2} 
implies that the stable splitting in~(\ref{Crabbsplitting}) is compatible with the inclusion 
\[\namedright{U(n+t)/U(t)}{}{U(r+n+t)/U(r+t)}.\] 
Thus as $n\rightarrow\infty$ the stable splitting of $U(\infty)/U(t)$ is compatible 
with that of $U(\infty)/U(r+t)$. Note that this includes the case for 
\(\namedright{U(\infty)}{}{U(\infty)/U(r)}\).  
Therefore, there is a stably homotopy commutative diagram 
\begin{equation} 
  \label{Utrstable} 
  \diagram 
      U(\infty)_{+}\rto\dto^{\simeq S} & U(\infty)/U(t)_{+}\rto\dto^{\simeq S} & U(\infty)/U(r+t)_{+}\dto^{\simeq S} \\ 
      \bigvee_{k=0}^{\infty} A_{k}\rto^-{\bigvee_{k=0}^{\infty} a_{k}} 
          & \bigvee_{k=0}^{\infty} A'_{k}\rto^-{\bigvee_{k=0}^{\infty} a'_{k}} & \bigvee_{k=0}^{\infty} A''_{k} 
  \enddiagram 
\end{equation} 
where $\simeq_{S}$ denotes a stable homotopy equivalence, 
\[\begin{split} 
    A_{k} & =\varinjlim_{m} 
           G_{k}(\mathbb{C}^{n}\oplus\mathbb{C}^{m})^{\mathfrak{u}(\zeta_{k})} \\  
     A'_{k} & =\varinjlim_{m} 
  G_{k}(\mathbb{C}^{n}\oplus\mathbb{C}^{m})^{\mathfrak{u}(\zeta_{k})\oplus(\zeta_{k}^{\ast}\otimes\mathbb{C}^{t})} \\  
     A''_{k} & =\varinjlim_{m} 
  G_{k}((\mathbb{C}^{r}\oplus\mathbb{C}^{n})\oplus\mathbb{C}^{m})^{\mathfrak{u}(\zeta_{k})\oplus(\zeta_{k}^{\ast}\otimes(\mathbb{C}^{r}\oplus\mathbb{C}^{t}))}, 
\end{split}\]  
the map $a_{k}$ is induced by the inclusion 
\(\namedright{\mathbb{C}^{0}}{}{\mathbb{C}^{0}\oplus\mathbb{C}^{t}}\), 
and the map $a'_{k}$ is induced by the inclusions 
\(\namedright{\mathbb{C}^{n}}{}{\mathbb{C}^{r}\oplus\mathbb{C}^{n}}\) 
and  
\(\namedright{\mathbb{C}^{t}}{}{\mathbb{C}^{r}\oplus\mathbb{C}^{t}}\). 

As noted by Miller~\cite[Theorem C]{M}, when $k=1$ the Thom space 
$G_{1}(\mathbb{C}^{n})^{\mathfrak{u}(\zeta_{1})\oplus(\zeta_{1}^{\ast}\otimes\mathbb{C}^{t})}$ 
in~(\ref{Crabbsplitting}) is identifiable. If $t=0$ it is homeomorphic to $S^{1}\vee\Sigma\mathbb{C}P^{n-1}$,
while if $t>1$ it is homeomorphic to the stunted projective space 
$\Sigma\mathbb{C}P^{n+t-1}/\Sigma\mathbb{C}P^{t}$, and the maps $a_{1}$ and $a'_{1}$ in~(\ref{Utrstable}) 
are the standard quotient maps.

\section{Stable splittings of $SU(4)$ and $SU(4)/SU(2)$} 
\label{sec:stable} 

The homotopy groups of spheres will play an important role in the next few sections. We follow Toda's notation~\cite{To} in all 
cases except one. Specifically, 
(i) for $n\geq 3$, $\eta_{n}=\Sigma^{n-3}\eta_{3}$ represents the 
generator of $\pi_{n+1}(S^{n})\cong\mathbb{Z}/2\mathbb{Z}$; 
(ii) for $n\geq 5$, $\nu_{n}=\Sigma^{n-3}\nu_{5}$ represents the 
generator of $\pi_{n+3}(S^{n})\cong\mathbb{Z}/24\mathbb{Z}$; 
and (iii) differing from Toda's notation, for $n\geq 3$, $\nu'_{n}=\Sigma^{n-3}\nu'_{3}$ 
represents the $n-3$ fold suspension of the generator~$\nu'_{3}$ of 
$\pi_{6}(S^{3})\cong\mathbb{Z}/12\mathbb{Z}$. Note that for $n\geq 5$, 
$\nu'_{n}=2\nu_{n}$. Compositions of these elements are denoted by juxtaposition, and we write $\eta_n^2=\eta_n\eta_{n+1}$, $\eta_n^3=\eta_n^2\eta_{n+2}$, and $\nu_n^2=\nu_n\nu_{n+3}$ in these cases. 

 
We will need some properties of $SU(4)$. There is an 
algebra isomorphism $\cohlgy{SU(4);\mathbb{Z}}\cong\Lambda(x,y,z)$, 
where the degrees of $x,y,z$ are $3,5,7$, respectively. This gives $\rcohlgy{SU(4);\mathbb{Z}}$ 
the module basis $\{x,y,z,xy,xz,yz,xyz\}$ in degrees $\{3,5,7,8,10,12,15\}$, and it follows that 
$SU(4)$ may be given a $CW$-structure with one cell in each of those dimensions. There is a canonical map 
\(\namedright{\Sigma\cpthree}{}{SU(4)}\) 
which induces a projection onto the generating set in cohomology. Notice 
that $\Sigma\cpthree$ is homotopy equivalent to the $7$-skeleton of $SU(4)$, 
and there is a homotopy cofibration 
\begin{equation} 
  \label{CP3cofib} 
  \llnameddright{S^{4}\vee S^{6}}{\eta_{3}\vee\nu'_{3}}{S^{3}}{}{\Sigma\cpthree}.  
\end{equation} 
Write $\cohlgy{SU(4);\mathbb{Z}/2\mathbb{Z}}\cong\Lambda(\bar{x},\bar{y},\bar{z})$ 
where $\bar{x}$, $\bar{y}$ and $\bar{z}$ are the respective mod-$2$ reductions of $x$, $y$ and $z$.
The action of the Steenrod algebra is determined on the generating set 
by $Sq^{2}(\bar{x})=\bar{y}$, $Sq^{2}(\bar{y})=Sq^{2}(\bar{z})=0$ and $Sq^{i}=0$ for 
all $i\neq 2$. In particular, we will later use the following facts derived from the 
Cartan formula: $Sq^{2}(\bar{x}\bar{y})=\bar{y}^{2}=0$, $Sq^{2}(\bar{x}\bar{z})=\bar{y}\bar{z}$ 
and $Sq^{4}(\bar{x}\bar{y})=0$. Thus $Sq^{2}$ and $Sq^{4}$ act trivially on 
$H^{9}(\Sigma SU(4);\mathbb{Z}/2\mathbb{Z})$ while $Sq^{2}$ acts nontrivially on 
$H^{11}(\Sigma SU(4);\mathbb{Z}/2\mathbb{Z})$.

The stable decomposition of $SU(4)$ has the following form. Regard $SU(4)$ as $U(4)/U(1)$, 
so~(\ref{Crabbsplitting}) applies. Kitchloo~\cite[Theorem C]{K} calculated the 
$\widetilde{E}^{\ast}$ cohomology of the stable summands for any complex-oriented 
cohomology theory $E$. In the case of ordinary cohomology with integer coefficients, 
with $\cohlgy{SU(4)}\cong\Lambda(x,y,z)$ 
for $\vert x\vert=3$, $\vert y\vert=5$ and $\vert z\vert=7$, he obtains 
\[\begin{split} 
    \rcohlgy{G_{1}(\mathbb{C}^{4})^{\mathfrak{u}(\zeta_{1})\oplus(\zeta_{1}^{\ast}\otimes\mathbb{C})}} 
       &  \cong\mathbb{Z}\{x,y,z\} \\  
   \rcohlgy{G_{2}(\mathbb{C}^{4})^{\mathfrak{u}(\zeta_{2})\oplus(\zeta_{2}^{\ast}\otimes\mathbb{C})}} 
       & \cong\mathbb{Z}\{xy,xz,yz\} \\   
   \rcohlgy{G_{3}(\mathbb{C}^{4})^{\mathfrak{u}(\zeta_{3})\oplus(\zeta_{3}^{\ast}\otimes\mathbb{C})}} 
        & \cong\mathbb{Z}\{xyz\}. 
\end{split}\] 
In terms of $CW$-complexes, it was already mentioned that 
$G_{1}(\mathbb{C}^{4})^{\mathfrak{u}(\zeta_{1})\oplus(\zeta_{1}^{\ast}\otimes\mathbb{C})} 
     \cong\Sigma\mathbb{C}P^{3}$. 
The second summand $N=G_{2}(\mathbb{C}^{4})^{\mathfrak{u}(\zeta_{2})\oplus(\zeta_{2}^{\ast}\otimes\mathbb{C})}$ is a $3$-cell complex with cells in dimensions $8$, $10$ and $12$. 
Finally, 
 the inclusion 
\(\namedright{S^{15}}{}{G_{3}(\mathbb{C}^{4})^{\mathfrak{u}(\zeta_{3})\oplus(\zeta_{3}^{\ast}\otimes\mathbb{C})}}\) 
of the bottom cell induces an isomorphism in cohomology and so is a homotopy 
equivalence by Whitehead's Theorem. Collecting this together gives the following. 

\begin{theorem} 
   \label{Miller} 
   There is a stable homotopy equivalence 
   \[SU(4)\simeq_{S}\Sigma\cpthree\vee N\vee S^{15}\] 
   where $N$ is a $3$-cell complex with cells in dimensions $8$, $10$ and $12$.   
\end{theorem} 
\vspace{-1cm}~$\qqed$\medskip 

In what follows, we will use Theorem~\ref{Miller} to produce a potentially different 
stable homotopy equivalence for $SU(4)$ that is compatible with the quotient map 
\(\namedright{SU(4)}{q}{S^{5}\times S^{7}}\), 
and better identify the space $N$. 

Define the space $C$ and maps $j$ and $\delta$ by the homotopy cofibration 
\begin{equation} 
  \label{Ccofib} 
  \namedddright{SU(4)}{q}{S^{5}\times S^{7}}{j}{C}{\delta}{\Sigma SU(4)}. 
\end{equation}  
Since $q^{\ast}$ is an inclusion onto the subalgebra $\Lambda(y,z)$ 
of $\Lambda(x,y,z)\cong\cohlgy{SU(4);\mathbb{Z}}$, the long exact sequence 
in cohomology induced by the cofibration sequence~(\ref{Ccofib}) implies that 
a module basis for $\rcohlgy{C;\mathbb{Z}}$ 
is given by $\{\sigma x,\sigma xy, \sigma xz, \sigma xyz\}$ 
in degrees $\{4,9,11,16\}$ respectively, where the elements of 
$\rcohlgy{C;\mathbb{Z}}$ have been identified with the image of $\delta^{\ast}$. 
So as a $CW$-complex, $C$ has one cell in each of the dimensions 
$\{4,9,11,16\}$. We will determine the stable homotopy type of $C$, the 
stable class of the map $j$, and a stable decomposition for the spaces and map 
\(\namedright{SU(4)}{q}{S^{5}\times S^{7}}\). 

From the left square in~(\ref{Utrstable}) we obtain a homotopy cofibration diagram 
\begin{equation} 
  \label{Aidgrm} 
  \diagram 
      U(\infty)_{+}\rto\dto^{\simeq_{S}} & U(\infty)/U(2)_{+}\rto\dto^{\simeq_{S}} & D\dto^{\simeq_{S}} \\ 
      \bigvee_{k=0}^{\infty} A_{k}\rto^-{\bigvee_{k=0}^{\infty} a_{k}} 
           & \bigvee_{k=0}^{\infty} A'_{k}\rto^-{\bigvee_{k=0}^{\infty} b_{k}} 
           & \bigvee_{k=0}^{\infty} B_{k} 
  \enddiagram 
\end{equation} 
that defines the spaces $D$ and $B_{k}$ and the maps $b_{k}$. In particular, the descriptions 
of $A_{1}$, $A'_{1}$ and the map $a_{1}$ imply that $B_{1}\simeq\Sigma(S^{1}\vee S^{3})$. 
Standard group homomorphisms give a commutative diagram 
\begin{equation} 
  \label{Uinfdgrm} 
  \diagram 
      SU(2)\rto\ddouble & SU(4)\rto^-{q}\dto & SU(4)/SU(2)\dto \\ 
      SU(2)\rto\dto & SU(\infty)\rto\dto & SU(\infty)/SU(2)\dto^{\cong} \\ 
      U(2)\rto & U(\infty)\rto & U(\infty)/U(2)  
  \enddiagram 
\end{equation} 
where the rows are homotopy fibration sequences. Combining~(\ref{Aidgrm}) and~(\ref{Uinfdgrm}) 
and rewriting $SU(4)/SU(2)$ as $S^{5}\times S^{7}$ gives a stably homotopy commutative diagram 
\begin{equation} 
  \label{CA''} 
  \diagram 
      SU(4)\rto^-{q}\dto & S^{5}\times S^{7}\rto^-{j}\dto & C\dto \\ 
      U(\infty)_{+}\rto\dto^{\simeq_{S}} & U(\infty)/U(2)_{+}\rto\dto^{\simeq_{S}} & D\dto^{\simeq_{S}} \\ 
      \bigvee_{k=0}^{\infty} A_{k}\rto^-{\bigvee_{k=0}^{\infty} a_{k}} 
           & \bigvee_{k=0}^{\infty} A'_{k}\rto^-{\bigvee_{k=0}^{\infty} b_{k}} 
           & \bigvee_{k=0}^{\infty} B_{k} 
  \enddiagram 
\end{equation}  

We draw two consequences from~(\ref{CA''}): Lemmas~\ref{stableC1} and~\ref{S57} below. First, let 
\(\imath\colon\namedright{S^{4}}{}{C}\) 
be the inclusion of the bottom cell. Since $S^{4}$ stably retracts off $B_{1}$, there 
is a stable composite 
\(\jmath\colon\namedddright{C}{}{D}{\simeq_{S}}{\bigvee_{k=1}^{\infty} B_{k}}{\mbox{pinch}}{B_{1}} 
      \longrightarrow S^{4}\).  

\begin{lemma} 
   \label{stableC1} 
   The map $\jmath$ is a stable left homotopy inverse for $\imath$, implying that there is a stable 
   homotopy equivalence $C\simeq_{S} S^{4}\vee C'$ where $C'$ is the homotopy cofibre of $\imath$. 
\end{lemma} 

\begin{proof} 
It is straightforward to see that the map 
\(\namedright{C}{}{D}\) 
induces an isomorphism in degree~$4$ cohomology. Thus $\jmath\circ\imath$ induces an 
isomorphism in homology and so is a homotopy equivalence. Thus $C\simeq_{S} S^{4}\vee C'$ 
where $C'$ is the homotopy cofibre of $\imath$. 
\end{proof} 

In general, for a $CW$-complex $X$ and a positive integer~$m$, let~$X_{m}$ be 
the $m$-skeleton of $X$. 

\begin{lemma} 
   \label{stableC11} 
   There is a stable homotopy equivalence $C_{11}\simeq S^{4}\vee S^{9}\vee S^{11}$. 
   Further, this equivalence can be chosen so that there is a homotopy commutative square 
   \[\diagram 
        C_{11}\rto\dto^{\simeq} & C\rto^-{\jmath} & S^{4}\ddouble \\ 
        S^{4}\vee S^{9}\vee S^{12}\rrto^-{\mbox{pinch}} & & S^{4}. 
      \enddiagram\] 
\end{lemma} 

\begin{proof} 
The description of $\rcohlgy{C;\mathbb{Z}}$ implies that $\rcohlgy{C';\mathbb{Z}}$ has 
$\mathbb{Z}$-generators in dimensions $9$, $11$ and~$16$. Therefore, as 
$\pi_{10}(S^{9})\cong\mathbb{Z}/2\mathbb{Z}$ and is generated by $\eta_{9}$, the $11$-skeleton 
of $C'$ is either $S^{9}\vee S^{11}$ or $\Sigma^{7}\mathbb{C}P^{2}$. The class $\eta_{9}$ is  
detected by the Steenrod operation $Sq^{2}$ in mod-$2$ cohomology. However, 
this acts trivially on $H^{9}(C';\mathbb{Z}/2\mathbb{Z})$ since the (stable) composite 
\(\nameddright{C'}{}{C}{\delta}{\Sigma SU(4)}\) 
induces an isomorphism on $H^{9}$ and $H^{11}$ and $Sq^{2}$ acts trivially 
on $H^{9}(\Sigma SU(4);\mathbb{Z}/2\mathbb{Z})$. Thus the $11$-skeleton of $C'$ is 
homotopy equivalent to $S^{9}\vee S^{11}$. Therefore Lemma~\ref{stableC1} implies 
that the $11$-skeleton of $C$ is stably homotopy equivalent to $S^{4}\vee S^{9}\vee S^{11}$. 

Next, let $F$ be the homotopy fibre of the composite 
\(\nameddright{C_{11}}{}{C}{\jmath}{S^{4}}\). 
Since this composite is degree~$1$ on the bottom cell, a cohomology Serre spectral sequence calculation 
shows that the $11$-skeleton of $F$ is $S^{9}\cup_{g} e^{11}$ for some attaching map $g$, 
and the restriction 
\(h\colon F_{11}\hookrightarrow\namedright{F}{}{C_{11}}\) 
is an epimorphism in integral and mod-$2$ cohomology. The map $g$ represents a class in 
$\pi_{10}(S^{9})\cong\mathbb{Z}/2\mathbb{Z}$, whose generator is detected by $Sq^{2}$. Since $h$ 
is an epimorphism in mod-$2$ cohomology, and $Sq^{2}$ vanishes in $C_{11}$ because it 
is homotopy equivalent to a wedge of spheres, $Sq^{2}$ also vanishes in $F_{11}$. 
Thus $F_{11}\simeq S^{9}\vee S^{11}$. Therefore the wedge sum 
\(\namedright{S^{4}\vee (S^{9}\vee S^{11})\simeq S^{4}\vee F_{11}}{\imath+h}{C_{11}}\) 
induces an isomorphism in cohomology and so is a homotopy equivalence. Since $h$ 
factors through the homotopy fibre of $\jmath$, the asserted homotopy commutative 
diagram exists. 
\end{proof} 

The second item to address from~(\ref{CA''}) is as follows. 
While it is well known that $S^{5}\times S^{7}$ is stably equivalent to 
$S^{5}\vee S^{7}\vee S^{12}$, we wish to make a particular choice of equivalence 
related to~(\ref{CA''}). The stable composite 
\(u\colon\nameddright{S^{5}\times S^{7}}{}{U(\infty)/U(2)_{+}}{\simeq_{S}}{\bigvee_{k=0}^{\infty} A'_{k}}\) 
factors through the $12$-skeleton of the range. Since 
$A_{1}'\cong\Sigma\mathbb{C}P^{\infty}/\Sigma\mathbb{C}P^{1}$, 
its $12$-skeleton is $\Sigma\mathbb{C}P^{5}/\Sigma\mathbb{C}P^{1}$. The 
cohomological calculations in~\cite[Theorem C]{K} imply that the $12$-skeleton of $(A'_{2})$ 
is $S^{12}$ while the $12$-skeleton of $A'_{k}$ for $k\geq 3$ is contractible. Thus~$u$ 
factors to give a homotopy commutative diagram
   \[\diagram 
        S^{5}\times S^{7}\rrto\dto^{u'} & & U(\infty)/U(2)_{+}\dto^{\simeq_{S}} \\ 
        (\Sigma\mathbb{C}P^{5}/\Sigma\mathbb{C}P^{1})\vee S^{12}\rto 
            & A'_{1}\vee A'_{2}\rto & \bigvee_{k=0}^{\infty} A'_{k} 
      \enddiagram\] 
for some map $u'$ that induces an epimorphism in cohomology. Note that 
$\Sigma\mathbb{C}P^{3}/\Sigma\mathbb{C}P^{1}\simeq S^{5}\vee S^{7}$ 
so composing into $\Sigma\mathbb{C}P^{5}/\Sigma\mathbb{C}P^{1}$ gives a map 
\(i'\colon\namedright{S^{5}\vee S^{7}}{}{\Sigma\mathbb{C}P^{5}/\Sigma\mathbb{C}P^{1}}\). 
The following lemma shows that $u'$ lifts through $i'\vee 1$. 

\begin{lemma} 
   \label{S57} 
   There is a stably homotopy commutative diagram 
   \[\diagram 
        S^{5}\times S^{7}\rdouble\dto^{\varepsilon} 
            & S^{5}\times S^{7}\rrto\dto^{u'} & & U(\infty)/U(2)_{+}\dto^{\simeq_{S}} \\ 
        (S^{5}\vee S^{7})\vee S^{12}\rto^-{i'\vee 1} 
            & (\Sigma\mathbb{C}P^{5}/\Sigma\mathbb{C}P^{1})\vee S^{12}\rto 
            & A'_{1}\vee A'_{2}\rto & \bigvee_{k=0}^{\infty} A'_{k} 
      \enddiagram\] 
   for some map $\varepsilon$, which is a homotopy equivalence. 
\end{lemma} 

\begin{proof} 
Observe that there is a homotopy cofibration 
\(\nameddright{S^{5}\vee S^{7}}{i'}{\Sigma\mathbb{C}P^{5}/\Sigma\mathbb{C}P^{1}}{}{S^{9}\vee S^{11}}\). 
So to show the existence of the lift of $u'$ it is equivalent to show that, stably, the composite 
\(\namedddright{S^{5}\times S^{7}}{u'}{(\Sigma\mathbb{C}P^{5}/\Sigma\mathbb{C}P^{1})\vee S^{12}} 
     {p_{1}}{\Sigma\mathbb{C}P^{5}/\Sigma\mathbb{C}P^{1}}{}{S^{9}\vee S^{11}}\) 
is null homotopic, where $p_{1}$ is the pinch map. 

Consider the diagram 
\[\diagram 
      U(4)/U(2)\rto & U(\infty)/U(2)\rto\dto^{\simeq S} & U(\infty)/U(4)\dto^{\simeq S} \\ 
      & \bigvee_{i=1}^{\infty} A'_{k}\rto^-{\bigvee_{k=1}^{\infty} a'_{k}}\dto^{p_{1}} 
            & \bigvee_{k=1}^{\infty} A''_{k}\dto^{p_{1}} \\ 
      & A'_{1}\rto^-{a'_{1}} & A''_{1} 
  \enddiagram\] 
where $p_{1}$ is the pinch map to the first wedge summand. The upper square stably homotopy 
commutes as it is the $r=t=2$ case of the right square in~(\ref{Utrstable}). The lower square 
commutes by the naturality of the pinch map. The top row is a fibration sequence, 
so the composite following the upper way around the diagram is null homotopic. Note that  
$U(4)/U(2)\simeq SU(4)/SU(2)\simeq S^{5}\times S^{7}$. Restricting to $12$-skeletons 
and noting that the restriction of $a'_{1}$ to $12$-skeletons is the collapse map  
\(\namedright{\Sigma\mathbb{C}P^{5}/\Sigma\mathbb{C}P^{1}}{} 
      {\Sigma\mathbb{C}P^{6}/\Sigma\mathbb{C}P^{4}\simeq S^{9}\vee S^{11}}\), 
the lower way around the diagram implies that the composite 
\(\namedddright{S^{5}\times S^{7}}{u'}{(\Sigma\mathbb{C}P^{5}/\Sigma\mathbb{C}P^{1})\vee S^{12}} 
     {p_{1}}{\Sigma\mathbb{C}P^{5}/\Sigma\mathbb{C}P^{1}}{}{S^{9}\vee S^{11}}\) 
is null homotopic, as required. 

Finally, since $u'$ induces an epimorphism in cohomology, do does the lift $\varepsilon$. But 
this implies that $\varepsilon$ induces an isomorphism in cohomology, and therefore an isomorphism in homology
by the universal coefficient theorem. Hence $\varepsilon$ is a 
homotopy equivalence by Whitehead's Theorem.
\end{proof} 

Next we consider the map 
\(\namedright{S^{5}\times S^{7}}{j}{C}\). 
Since $S^{5}\times S^{7}$ has dimension $12$, the map $j$ factors through the $12$-skeleton 
of $C$. Since $C$ has cells in dimensions $4$, $9$, $11$ and $16$, there is a homotopy  
equivalence $C_{12}\simeq C_{11}$. Thus $j$ factors as a composite 
\[\nameddright{S^{5}\times S^{7}}{j'}{C_{11}}{}{C}\] 
for some map $j'$. 

\begin{lemma} 
   \label{stablej1} 
   Stably, there is a homotopy commutative diagram 
   \[\diagram 
         S^{5}\times S^{7}\rto^-{j'}\dto^{\simeq_{S}} & C_{11}\dto^{\simeq_{S}} \\  
         S^{5}\vee S^{7}\vee S^{12}\rto^{\overline{j}}  
                 & S^{4}\vee S^{9}\vee S^{11}  
     \enddiagram\] 
   where $\overline{j}$ is the wedge sum of (i) 
   \(\llnamedright{S^{5}\vee S^{7}}{\eta_{4}+\nu'_{4}}{S^{4}}\) 
   and (ii) \(\llnamedright{S^{12}}{s\cdot\nu'_{9}+\eta_{11}}{S^{9}\vee S^{11}}\) 
   for some $s\in\mathbb{Z}/12\mathbb{Z}$.  
\end{lemma} 

\begin{proof} 
From the map 
\(\namedright{\Sigma\mathbb{C}P^{3}}{}{SU(4)}\) 
we obtain a homotopy commutative diagram 
\[\diagram 
      S^{3}\rto\dto & \Sigma\mathbb{C}P^{3}\rto\dto & S^{5}\vee S^{7}\dto^{i} \\ 
      SU(2)\rto & SU(4)\rto^-{q} & S^{5}\times S^{7} 
  \enddiagram\] 
where the top row is a cofibration sequence, the bottom row is a fibration sequence and $i$ is the inclusion. 
Starting from the right square, take homotopy cofibres horizontally. Recalling the homotopy 
cofibration in~(\ref{CP3cofib}) we obtain a homotopy cofibration diagram 
\[\diagram 
      \Sigma\mathbb{C}P^{3}\rto\dto & S^{5}\vee S^{7}\rto^-{\eta_{4}+\nu'_{4}}\dto & S^{4}\dto^{\iota} \\ 
      SU(4)\rto^-{q} & S^{5}\times S^{7}\rto^-{j} & C  
  \enddiagram\] 
For the remainder of the proof we work in the stable category (but retain unstable indexing 
for maps $\eta_{4}$, $\nu'_{4}$ to mesh better with later reference). There is a homotopy 
equivalence $S^{5}\times S^{7}\simeq S^{5}\vee S^{7}\vee S^{12}$. 
Since $\iota$ has a left homotopy inverse by Lemma~\ref{stableC1},
the restriction of $j$ to $S^{5}\vee S^{7}$ is homotopic to $\eta_{4}+\nu'_{4}$. 
As $j$ maps $S^5\times S^7$ to the $12$-skeleton of $C$, which is homotopy 
equivalent to the $11$-skeleton, on the top cell $j$ is a map 
\(\gamma\colon\namedright{S^{12}}{}{C_{12}\simeq_{S} S^{4}\vee S^{9}\vee S^{11}}\).
Furthermore, $\gamma$ is homotopic to $\gamma_{4}+\gamma_{9}+\gamma_{11}$ 
where $\gamma_{i}$ is~$\gamma$ composed with the pinch map to $S^{i}$. 
We now identify $\gamma_{4}$, $\gamma_{9}$ and $\gamma_{11}$. 

For $\gamma_{4}$, consider the diagram 
\[\spreaddiagramrows{-0.1pc} 
  \diagram 
     & S^{5}\times S^{7}\rto^-{\gamma}\ddouble 
          & C_{12}\rto^-{\simeq}\dto & S^{4}\vee S^{9}\vee S^{11}\xto[dddd]^{\mbox{pinch}} \\ 
     & S^{5}\times S^{7}\rto^-{j}\dto\xto[ddl]_{\varepsilon} & C\dto & \\ 
     & U(\infty)/U(2)_{+}\rto\dto^{\simeq_{S}} & D\dto & \\ 
     (S^{5}\vee S^{7})\vee S^{12}\rto\dto^{q_{1}} 
           & \bigvee_{k=0}^{\infty} A'_{k}\rto^-{\bigvee_{k=0}^{\infty} b_{k}}\dto^{q_{1}} 
          & \bigvee_{k=0}^{\infty} B_{k}\dto^{q_{1}} & \\ 
      S^{5}\vee S^{7}\rto & A'_{1}\rto^-{b_{1}} & B_{1}\rto & S^{4}.    
  \enddiagram\] 
The left triangle homotopy commutes by Lemma~\ref{S57}. The lower left square 
commutes since the map 
\(\namedright{(S^{5}\vee S^{7})\vee S^{12}}{}{\bigvee_{k=0}^{\infty} A'_{k}}\) 
factors through the inclusion of $A'_{1}\vee A'_{2}$ and the pinch map $q_1$ is natural.
The upper middle square homotopy commutes by definition 
of $\gamma$. The middle column otherwise homotopy commutes by~(\ref{CA''}), 
and the right rectangle homotopy commutes by Lemma~\ref{stableC11}. Precomposing 
with the composite 
\(S^{12}\hookrightarrow\namedright{S^{5}\vee S^{7}\vee S^{12}}{\varepsilon^{-1}}{S^{5}\times S^{7}}\), 
the lower direction around the diagram is null homotopic, while the upper direction around 
the diagram is the definition of $\gamma_{4}$. Hence $\gamma_{4}$ is null homotopic. 

Observe that $\gamma_{9}$ and $\gamma_{11}$ are multiples of the generators 
$\nu_{9}$ and $\eta_{9}$ respectively. The $2$-primary component of the class $\nu_{9}$ 
and the whole class $\eta_{9}$ are detected by the Steenrod operations $Sq^{4}$ 
and $Sq^{2}$ respectively. As $\gamma$ is a restriction of $j$, the detection is 
determined by the Steenrod operations in $\cohlgy{\Sigma SU(4);\mathbb{Z}/2\mathbb{Z}}$. 
But $Sq^{2}$ and $Sq^{4}$ act trivially on $H^{9}$ while $Sq^{2}$ acts nontrivially on $H^{11}$. 
As~$\nu_{9}$ has order $24$, this implies that $\gamma_{9}\simeq 2s\cdot\nu_{9}$ for some 
$s\in\mathbb{Z}/24\mathbb{Z}$ and $\gamma_{11}\simeq\eta_{11}$. 
As $2\cdot\nu_{9}\simeq\nu'_{9}$, we obtain $\gamma_{9}\simeq s\cdot\nu'_{9}$ where 
we may now regard $s$ as an element of $\mathbb{Z}/12\mathbb{Z}$.  
Thus the restriction of $j$ to $S^{12}$ is $s\cdot\nu'_{9}+\eta_{11}$. 
\end{proof} 

Next, we return to the stable decomposition of $C$, which will use Lemma~\ref{stablej1}. 

\begin{lemma} 
   \label{stableC} 
   There is a stable homotopy equivalence $C\simeq_{S} S^{4}\vee S^{9}\vee S^{11}\vee S^{16}$. 
\end{lemma} 

\begin{proof} 
Throughout the proof we work in the stable category. 
By Lemma~\ref{stableC11}, there is a homotopy equivalence 
$C_{11}\simeq S^{4}\vee S^{9}\vee S^{11}$. Therefore there is a homotopy cofibration 
\[\nameddright{S^{15}}{f}{S^{4}\vee S^{9}\vee S^{11}}{}{C}\] 
where $f$ attaches the top cell of $C$. We have $f\simeq f_{4}+f_{9}+f_{11}$, where $f_{i}$ is $f$ 
composed with the pinch map to $S^{i}$. Since $S^{4}$ retracts off $C$, $f_{4}$ is null homotopic. 
Since the stable $4$-stem is zero, $f_{11}$ is null homotopic. Since the stable 
$6$-stem is $\mathbb{Z}/2\mathbb{Z}$, generated by $\nu^{2}$, we have $f_{9}\simeq t\cdot\nu^{2}_{9}$ 
for some $t\in\mathbb{Z}/2\mathbb{Z}$. We will show that $f_{9}$ is null homotopic. If so, then $f$ 
is null homotopic and the asserted homotopy equivalence for $C$ follows. 

It remains to show that $f_{9}$ is null homotopic. Since $S^{5}\times S^{7}$ is $12$-dimensional, the map 
\(\namedright{S^{5}\times S^{7}}{j}{C}\) 
factors through the $12$-skeleton $C_{12}$ of $C$. From this we obtain a homotopy pushout diagram 
\[\diagram 
     & S^{15}\rdouble\dto^{f} & S^{15}\dto^{\delta'\circ f} \\ 
     S^{5}\times S^{7}\rto^-{j'}\ddouble & C_{12}\rto^-{\delta'}\dto & (\Sigma SU(4))_{13}\dto \\ 
     S^{5}\times S^{7}\rto^-{j} & C\rto^-{\delta} & \Sigma SU(4) 
 \enddiagram\] 
where $j'$ is the factorization of $j$ through the $12$-skeleton (which is homotopy equivalent 
to $C_{11}$ and therefore consistent with the notation in Lemma~\ref{stablej1}) and $\delta'$ is the 
restriction of $\delta$. In particular, this identifies the attaching map for the top cell of 
$\Sigma SU(4)$ as $\delta'\circ f$. By Theorem~\ref{Miller}, the top cell of $SU(4)$ retracts off, 
so $\delta'\circ f$ is null homotopic. Therefore $f$ lifts through~$j'$. 
Using the homotopy equivalences
$S^{5}\times S^{7}\simeq S^{5}\vee S^{7}\vee S^{12}$ 
and, from Lemma~\ref{stableC11}, $C_{12}\simeq C_{11}\simeq S^{4}\vee S^{9}\vee S^{11}$,  
we obtain a homotopy commutative diagram
\[\diagram 
     & S^{5}\vee S^{7}\vee S^{12}\rto^-{p}\dto^{j'} & S^{12}\dto^{g} \\ 
     S^{15}\rto^-{f}\urto^-{f'} & S^{4}\vee S^{9}\vee S^{11}\rto^-{p'} & S^{9} 
  \enddiagram\] 
where $f'$ is a lift of $f$ through $j'$, $p$ and $p'$ are pinch maps to wedge summands, 
$g$ is the restriction of $p'\circ j'$ to $S^{12}$, and the right square homotopy commutes by 
connectivity. The lower row 
is the definition of $f_{9}$. Thus $f_{9}\simeq g\circ p\circ f'$. The composite $p\circ f'$ is 
some multiple of the class~$\nu_{12}$. By Lemma~\ref{stablej1}, the restriction 
of $j'$ to $S^{12}$ is $s\cdot\nu'_{9}+\eta_{11}$; therefore $g\simeq s\cdot\nu'_{9}$. Thus 
$f_{9}\simeq g\circ p\circ f'$ is some multiple of $\nu'_{9}\circ\nu_{12}$. But $\nu'_{9}\simeq 2\nu_{9}$ 
and $\nu^{2}_{9}$ has order $2$, so $f_{9}$ is null homotopic, as required. 
\end{proof} 

Combining Lemmas~\ref{stablej1} and~\ref{stableC} gives the following. 

\begin{proposition} 
   \label{Cjstable} 
   Stably, there is a homotopy commutative diagram 
   \[\diagram 
         S^{5}\times S^{7}\rto^-{j}\dto^{\simeq_{S}} & C\dto^{\simeq_{S}} \\  
         S^{5}\vee S^{7}\vee S^{12}\rto^{\overline{j}}  
                 & S^{4}\vee S^{9}\vee S^{11}\vee S^{16} 
     \enddiagram\] 
   where $\overline{j}$ is the wedge sum of (i) 
   \(\llnamedright{S^{5}\vee S^{7}}{\eta_{4}+\nu'_{4}}{S^{4}}\)  
   and (ii) \(\llnamedright{S^{12}}{s\cdot\nu'_{9}+\eta_{11}}{S^{9}\vee S^{11}}\) 
   for some $s\in\mathbb{Z}/12\mathbb{Z}$. 
\end{proposition} 
\vspace{-5mm}$\qqed$ 
 
Define $M$ by the homotopy cofibration 
\[\llnameddright{S^{11}}{s\cdot\nu'_{8}+\eta_{10}}{S^{8}\vee S^{10}}{}{M}\] 
where $s$ is as in Proposition~\ref{Cjstable}. 

\begin{theorem} 
   \label{Millernat2} 
   Stably, there is a homotopy commutative diagram 
   \[\diagram 
         SU(4)\rto^{q}\dto^{\simeq_{S}} & S^{5}\times S^{7}\dto^-{\simeq_{S}} \\ 
          \Sigma\cpthree\vee M\vee S^{15}\rto^{\overline{q}}  
              & S^{5}\vee S^{7}\vee S^{12}.  
     \enddiagram\] 
   where $\overline{q}$ is the wedge sum of: (i) the map  
   \(\namedright{\Sigma\cpthree}{}{S^{5}\vee S^{7}}\) 
   that collapses the bottom cell, (ii) the pinch map 
   \(\lnamedright{M}{}{S^{12}}\) 
   to the top cell, and (iii) the trivial map 
   \(\namedright{S^{15}}{}{\ast}\).~$\qqed$  
\end{theorem} 

\begin{proof} 
Recall that there are homotopy cofibration sequences 
\(\namedddright{SU(4)}{q}{S^{5}\times S^{7}}{j}{C}{\delta}{\Sigma SU(4)}\), and from~(\ref{CP3cofib}), \(\llnameddright{S^{6}}{\nu'_{3}+\eta_{5}}{S^{3}\vee S^{5}}{}{\Sigma\mathbb{C}P^{3}}\). 
Together with the definition of $M$, from Proposition~\ref{Cjstable} we obtain a homotopy commuting diagram in the stable category
\[\diagram 
      S^{5}\times S^{7}\rto^-{j}\dto^{\simeq_{S}} & C\rto^-{\delta}\dto^{\simeq_{S}} 
           & \Sigma SU(4)\rto^-{\Sigma q}\dto^{\theta} & \Sigma(S^{5}\times S^{7})\dto^{\simeq_{S}} \\  
      S^{5}\vee S^{7}\vee S^{12}\rto^{\overline{j}} 
           & S^{4}\vee S^{9}\vee S^{11}\vee S^{16}\rto 
           & \Sigma^{2}\mathbb{C}P^{3}\vee\Sigma M\vee S^{16}\rto^-{\widetilde{q}} 
           & S^{6}\vee S^{8}\vee S^{13} 
  \enddiagram\] 
where the rows are cofibration sequences, $\theta$ is an induced map of cofibres, and $\widetilde{q}$ collapses out the 
bottom cell of $\Sigma^{2}\mathbb{C}P^{3}$, pinches $\Sigma M$ to its top cell and 
collapses out the $S^{16}$. The Five Lemma implies that $\theta$ induces 
an isomorphism in homology and so is a stable homotopy equivalence. Each of the 
maps describing $\widetilde{q}$ desuspends, so $\widetilde{q}\simeq\Sigma\overline{q}$. 
Thus the right square is a suspension, and as we are working stably, it 
may be desuspended to give the asserted homotopy commutative diagram. 
\end{proof} 

\begin{remark} 
The stable homotopy equivalence for $SU(4)$ in Theorem~\ref{Millernat2} may be 
different from the one in Theorem~\ref{Miller}, in the sense that the maps realizing the 
decomposition may be non-homotopic. The space $M$ in Theorem~\ref{Millernat2} 
is homotopy equivalent to the space $N$ in Theorem~\ref{Miller}, as there is a stable 
map between them inducing an isomorphism in homology, so the description of $M$ 
as the homotopy cofibre of $s\cdot\nu'_{8}+\eta_{10}$ also describes $N$ more precisely. 
\end{remark}

\section{The triple suspension of $C$ and $j$} 
\label{sec:cofib} 

For the remainder of the paper all spaces and maps will be localized at $2$. This 
corresponds to the fact from Section~\ref{sec:initialbound} that we are reduced to proving the
$2$-primary statement in Theorem~\ref{looppartialorder2}. 

The stable decomposition of $C$ in Lemma~\ref{stableC} will be useful but we will 
ultimately need to work with unstable information in the form of the homotopy type 
of $\Sigma^{3} C$ and the homotopy class of $\Sigma^{3} j$. We start 
with the homotopy type of $\Sigma^{3} C$. The $CW$-structure for $C$ implies 
that there are homotopy cofibrations 
\begin{align} 
   \label{Ccofib1} 
   & \nameddright{S^{8}}{g_{1}}{S^{4}}{}{C_{9}} \\ 
   \label{Ccofib2} 
   & \nameddright{S^{10}}{g_{2}}{C_{9}}{}{C_{11}} \\ 
   \label{Ccofib3} 
   & \nameddright{S^{15}}{g_{3}}{C_{11}}{}{C} 
\end{align}
 
\begin{lemma} 
   \label{C9decomp} 
   There is a homotopy equivalence $\Sigma^{2} (C_{9})\simeq S^{6}\vee S^{11}$.   
\end{lemma} 

\begin{proof} 
By~\cite[Proposition 5.8]{To}, $\pi_{10}(S^{6})=0$, so the map $\Sigma^{2} g_{1}$ 
in~(\ref{Ccofib1}) is null homotopic. The asserted homotopy equivalence 
for $\Sigma^{2} (C_{9})$ follows immediately. 
\end{proof} 

\begin{lemma} 
   \label{C11decomp} 
   There is a homotopy equivalence  
   $\Sigma^{2} (C_{11})\simeq S^{6}\vee S^{11}\vee S^{13}$.   
\end{lemma} 

\begin{proof} 
Substituting the homotopy equivalence in Lemma~\ref{C9decomp} into the 
double suspension of~(\ref{Ccofib2}) gives a homotopy cofibration 
\(\nameddright{S^{12}}{\Sigma^{2} g_{2}}{S^{6}\vee S^{11}}{}{\Sigma^{2} (C_{11})}\). 
By the Hilton-Milnor Theorem, $\Sigma^{2} g_{2}\simeq a+b$ where 
$a$ and $b$ are obtained by composing $\Sigma^{2} g_{2}$ with the 
pinch maps to $S^{6}$ and $S^{11}$ respectively. We claim that each of $a$ 
and $b$ is null homotopic, implying that $\Sigma^{2} g_{2}$ is null 
homotopic, from which the asserted homotopy equivalence for $\Sigma^{2} (C_{11})$ 
follows immediately. 

By Lemma~\ref{stableC11}, $C_{11}$ is stably homotopy equivalent to a 
wedge of spheres. Thus $g_{2}$ is stably 
trivial, implying that $a$ and $b$ are as well. On 
the other hand, $a$ and $b$ are represented by classes in 
$\pi_{12}(S^{6})\cong\mathbb{Z}/2\mathbb{Z}$ 
and $\pi_{12}(S^{11})\cong\mathbb{Z}/2\mathbb{Z}$ respectively. 
By~\cite[Propositions~5.1 and~5.11]{To}, 
these groups are generated by $\nu_{6}^{2}$ and $\eta_{11}$, both of which 
are stable. Thus the only way that~$a$ and $b$ can be stably trivial is if 
both are already trivial. Hence $\Sigma^{2} g_{2}$ is null homotopic. 
\end{proof} 

\begin{lemma} 
   \label{Cdecomp} 
   There is a homotopy equivalence 
   $\Sigma^{3} C\simeq E\vee S^{12}\vee S^{14}$ 
   where $E$ is given by a homotopy cofibration 
   \(\llnameddright{S^{18}}{u\cdot\bar{\nu}_{7}\nu_{15}}{S^{7}}{}{E}\) 
   for some $u\in\mathbb{Z}/2\mathbb{Z}$. 
\end{lemma} 

\begin{proof} 
Substituting the homotopy equivalence in Lemma~\ref{C11decomp} into the 
double suspension of~(\ref{Ccofib3}) gives a homotopy cofibration 
\(\nameddright{S^{17}}{\Sigma^{2} g_{3}}{S^{6}\vee S^{11}\vee S^{13}}{}{\Sigma^{2} C}\). 
By the Hilton-Milnor Theorem, $\Sigma^{2} g_{3}\simeq a+b+c+d$ where 
$a$, $b$ and $c$ are obtained by composing $\Sigma^{2} g_{3}$ with the 
pinch maps to $S^{6}$, $S^{11}$ and~$S^{13}$ respectively, and $d$ is a composite 
\(\nameddright{S^{17}}{}{S^{16}}{w}{S^{6}\vee S^{11}\vee S^{13}}\). 
Here, $w$ is the Whitehead product of the identity maps on $S^{6}$ and $S^{11}$. 
As $\Sigma w$ is null homotopic, we instead consider  
\[\nameddright{S^{18}}{\Sigma^{3} g_{3}}{S^{7}\vee S^{12}\vee S^{14}} 
     {}{\Sigma^{3} C}\]   
where $\Sigma^{3} g_{3}\simeq\Sigma a+\Sigma b+\Sigma c$. 

By Lemma~\ref{stableC}, $C$ is stably homotopy equivalent to a wedge 
of spheres, so $\Sigma^{3} g_{3}\simeq\Sigma a+\Sigma b+\Sigma c$ is 
stably trivial. Thus each of $\Sigma a$, $\Sigma b$ 
and $\Sigma c$ is stably trivial. Observe that both $\Sigma b$ and $\Sigma c$ 
are in the stable range, impling that they are null homotopic. On the other hand, 
$\Sigma a$ represents a class in $\pi_{18}(S^{7})$. By~\cite[Theorem 7.4]{To},  
$\pi_{18}(S^{7})\cong\mathbb{Z}/8\mathbb{Z}\oplus\mathbb{Z}/2\mathbb{Z}$   
where the order $8$ generator is the stable classz~$\zeta_{7}$ and the 
order~$2$ generator is the unstable class $\bar{\nu}_{7}\nu_{15}$. Note 
too that the stable order of $\zeta_{7}$ is~$8$, so the only nontrivial unstable 
class in $\pi_{18}(S^{7})$ is $\bar{\nu}_{7}\nu_{15}$. As $\Sigma a$ is stably 
trivial, we obtain $\Sigma a=u\cdot\bar{\nu}_{7}\nu_{15}$ for some 
$u\in\mathbb{Z}/2\mathbb{Z}$. Hence $\Sigma^{2} g_{3}$ factors as the composite 
\(\llnamedright{S^{18}}{u\cdot\bar{\nu}_{7}\nu_{15}}{S^{7}}\hookrightarrow 
     S^{7}\vee S^{12}\vee S^{14}\), 
from which the asserted homotopy decomposition of $\Sigma^{3} C$ follows. 
\end{proof} 

Next, we identify $\Sigma^{3} j$. Let 
\[\iota\colon\namedright{S^{7}}{}{E}\] 
be the inclusion of the bottom cell. 

\begin{lemma} 
   \label{jclass} 
   There is a homotopy commutative diagram 
   \[\diagram 
          S^{8}\vee S^{10}\vee S^{15}\rto^{a+b+c} \dto^{\simeq}  
               & E\vee S^{12}\vee S^{14}\dto^-{\simeq} \\ 
          \Sigma^{3}(S^{5}\times S^{7})\rto^{\Sigma^{3} j}  
               & \Sigma^{3} C   
     \enddiagram\]   
   where $a$, $b$ and $c$ respectively are the composites 
   \begin{align*} 
          & a\colon\nameddright{S^{8}}{\eta_{7}}{S^{7}}{\iota}{E}\hookrightarrow 
                E\vee S^{12}\vee S^{14} \\  
          & b\colon\nameddright{S^{10}}{\nu'_{7}}{S^{7}}{\iota}{E}\hookrightarrow 
                E\vee S^{12}\vee S^{14} \\ 
          & c\colon\llllnameddright{S^{15}}{\psi+s\cdot\nu'_{12}+\eta_{14}}{S^{7}\vee S^{12}\vee S^{14}} 
                {\iota\vee 1\vee 1}{E\vee S^{12}\vee S^{14}} 
   \end{align*} 
   where $s$ is as in Proposition~\ref{Cjstable} and $\psi=t\cdot\sigma'\eta_{14}$ 
   for some $t\in\mathbb{Z}/2\mathbb{Z}$. 
\end{lemma} 

\begin{proof} 
By Proposition~\ref{Cjstable}, the diagram in the statement of the lemma 
stably homotopy commutes if~$c$ is replaced by the composite 
\(c'\colon\llllnameddright{S^{15}}{\ast+s\cdot\nu'_{12}+\eta_{14}}{S^{7}\vee S^{12}\vee S^{14}} 
                {\iota\vee 1\vee 1}{E\vee S^{12}\vee S^{14}}\). 
Since $a$ and $b$ are in the stable range, the diagram in the statement of the 
lemma therefore does homotopy commute when restricted to $S^{8}\vee S^{10}$. 
However, $c'$ is not in the stable range. It fails to be so only by a map 
\(\psi''\colon\namedright{S^{15}}{}{S^{7}}\). 
Thus if $c''$ is the composite 
\(c''\colon\llllnameddright{S^{15}}{\psi''+s\cdot\nu'_{12}+\eta_{14}}{S^{7}\vee S^{12}\vee S^{14}} 
                {\iota\vee 1\vee 1}{E\vee S^{12}\vee S^{14}}\) 
then the diagram in the statement of the lemma homotopy commutes 
with $c$ replaced by $c''$. 

More can be said. By~\cite[Theorem 7.1]{To} (stated later also in~(\ref{SU4groups})),  
$\pi_{15}(S^{7})\cong\mathbb{Z}/2\mathbb{Z}\oplus\mathbb{Z}/2\mathbb{Z}\oplus 
    \mathbb{Z}/2\mathbb{Z}$ 
with generators $\sigma'\nu_{14}$, $\bar{\nu}_{7}$ and $\epsilon_{7}$. 
Thus $\psi''=t\cdot\sigma'\nu_{14}+u\cdot\bar{\nu}_{7} + v\cdot\epsilon_{7}$ 
for some $t,u,v\in\mathbb{Z}/2\mathbb{Z}$. The generators $\bar{\nu}_{7}$ 
and $\epsilon_{7}$ are stable while $\sigma'\nu_{14}$ is unstable. So as 
$c''$ stabilizes to $c$, we must have $\psi''$ stabilizing to the trivial map. 
Thus $u$ and $v$ must be zero. Hence $\psi''=t\cdot\sigma'\nu_{14}$. 
Now $c''$ is exactly the map $c$ described in the statement of the lemma. 
\end{proof}

\section{Preliminary information on the homotopy groups of $SU(4)$} 
\label{sec:htpygroups} 

This section records some information on the homotopy groups 
of $SU(4)$ which will be needed subsequently. Consider the homotopy fibration 
\[\nameddright{S^{3}}{i}{SU(4)}{q}{S^{5}\times S^{7}}.\] 
This induces a long exact sequence of homotopy groups 
\[\cdots\longrightarrow\namedddright{\pi_{n+1}(S^{5}\times S^{7})}{} 
     {\pi_{n}(S^{3})}{i_{\ast}}{\pi_{n}(SU(4))}{q_{\ast}}{\pi_{n}(S^{5}\times S^{7})} 
     \longrightarrow\cdots\] 
Following~\cite{MT}, the notation $[\alpha\oplus\beta]\in\pi_{n}(SU(4))$ means that 
$[\alpha\oplus\beta]$ is an element of $\pi_{n}(SU(4))$ with the property 
that $q_{\ast}([\alpha\oplus\beta])=\alpha\oplus\beta$ for 
$\alpha\in\pi_{n}(S^{5})$ and $\beta\in\pi_{n}(S^{7})$. The homotopy groups 
of $SU(4)$ in low dimensions were determined by Mimura and Toda~\cite{MT}. 

The information presented will be split into two parts, the first corresponding 
to subsequent calculations involving $\pi_{m}(SU(4))$ for $m\in\{5,7,8,10\}$  
and the second corresponding to calculations involving $\pi_{15}(SU(4))$. 

First, for $r\geq 1$, let 
\(\underline{2}^{r}\colon\namedright{S^{7}}{}{S^{7}}\) 
be the map of degree $2^{r}$. In general, the degree two map on $S^{2n+1}$ 
need not induce multiplication by $2$ in homotopy groups. However, 
as $S^{7}$ is an $H$-space, the degree $2$ map on $S^{7}$ is homotopic 
to the $2^{nd}$-power map, implying that it does in fact induce multiplication 
by $2$ in homotopy groups. We record this for later use. 

\begin{lemma} 
   \label{S72} 
   The map 
   \(\namedright{S^{7}}{\underline{2}}{S^{7}}\) 
    induces multiplication by $2$ in homotopy groups.~$\qqed$ 
\end{lemma}

\subsection{Dimensions $5$, $7$, $8$ and $10$} 
The relevant table of homotopy groups from~\cite[Theorem 6.1]{MT} is: 
\begin{equation} 
  \label{SU4groups1} 
  \begin{tabular}{|c|c|c|c|c|}\hline 
          & $\pi_{5}(SU(4))$ & $\pi_{7}(SU(4))$ & $\pi_{8}(SU(4))$ & $\pi_{10}(SU(4))$ \\ \hline 
          $2$-component & $\mathbb{Z}$ & $\mathbb{Z}$ & $\mathbb{Z}/8\mathbb{Z}$ 
             & $\mathbb{Z}/8\mathbb{Z}\oplus\mathbb{Z}/2\mathbb{Z}$ \\ \hline 
          generators & $[\underline{2}\oplus\ast]$ & $[\eta_{5}^{2}\oplus\underline{2}]$ 
             & $[\nu_{5}\oplus\eta_{7}]$  
             & $[\nu_{7}]$, $[\nu_{5}\eta_{8}^{2}]$ \\ \hline 
  \end{tabular} 
\end{equation} 
In addition, Mimura and Toda~\cite[Lemma 6.2(i)]{MT} proved that 
\(\namedright{\pi_{n+1}(S^{5}\times S^{7})}{}{\pi_{n}(S^{3})}\) 
is an epimorphism for $n\in\{8,10\}$, implying the following.  

\begin{lemma} 
   \label{pi810inj} 
   The map 
   \(\namedright{\pi_{n}(SU(4))}{q_{\ast}}{\pi_{n}(S^{5}\times S^{7})}\) 
   is an injection for $n\in\{8,10\}$.~$\qqed$ 
\end{lemma} 

We record the following relations in the homotopy groups of spheres. 

\begin{lemma} 
   \label{Todarelns1} 
   The following hold: 
   \begin{letterlist} 
      \item $2\nu'_{3}\simeq\eta_{3}^{3}$; 
      \item $4\nu_{5}\simeq\eta_{5}^{3}$; 
      \item $\eta_{5}^{2}\nu'_{7}\simeq\ast$.  
   \end{letterlist} 
\end{lemma} 

\begin{proof} 
Part~(a) is by~\cite[Equation 5.3]{To}, part~(b) is by~\cite[Lemma 5.4]{To} together 
with part~(a), and part~(c) holds since $\nu'_{7}\simeq 2\nu_{7}$ by part~(b) while 
$\eta_{5}^{2}$ has order~$2$. 
\end{proof} 

For convenience, let 
\[d\colon\namedright{S^{7}}{}{SU(4)}\] 
represent the generator $[\eta_{5}^{2}\oplus\underline{2}]$ of $\pi_{7}(SU(4))$.  

\begin{lemma} 
   \label{S810dgrms} 
   There are homotopy commutative diagrams 
   \[\diagram 
          S^{8}\rto^-{[\nu_{5}\oplus\eta_{7}]}\dto^{\eta_{7}} & SU(4)\dto^{4} 
              & & S^{10}\rto^-{[\nu_{7}]}\dto^{\nu'_{7}} & SU^{4}\dto^{4} \\ 
          S^{7}\rto^-{d} & SU(4) & & S^{7}\rto^-{d} & SU(4). 
     \enddiagram\] 
\end{lemma} 

\begin{proof} 
By Lemma~\ref{pi810inj}, 
\(\namedright{\pi_{n}(SU(4))}{q_{\ast}}{\pi_{n}(S^{5}\times S^{7})}\) 
is an injection for $n\in\{8,10\}$. So in both cases it suffices to show that the 
asserted homotopies hold after composition with 
\(\namedright{SU(4)}{q}{S^{5}\times S^{7}}\). 
Since the composite 
\(\nameddright{S^{7}}{d}{SU(4)}{q}{S^{5}\times S^{7}}\) 
is $\eta_{5}^{2}\times\underline{2}$, the two assertions will follow if we prove: 
 
(i) $(\eta_{5}^{2}\times\underline{2})\circ\eta_{7}\simeq q\circ 4\circ[\nu_{5}\oplus\eta_{7}]$; 

(ii) $(\eta_{5}^{2}\times\underline{2})\circ\nu'_{7}\simeq q\circ 4\circ[\nu_{7}]$. 

By Lemma~\ref{S72}, $\underline{2}\circ\eta_{7}\simeq 2\eta_{7}$ and 
$\underline{2}\circ\nu'_{7}\simeq 2\nu'_{7}$. Since $\eta_{7}$ has order~$2$ 
we obtain $\underline{2}\circ\eta_{7}\simeq\ast$. By Lemma~\ref{Todarelns1}~(a) and~(c),  
$2\nu'_{7}\simeq\eta_{7}^{3}$ and $\eta_{5}^{2}\nu'_{7}\simeq\ast$. Thus (i) and (ii) reduce to proving: 

(i$^{\prime}$) $\eta_{5}^{3}\simeq q\circ 4\circ[\nu_{5}\oplus\eta_{7}]$; 

(ii$^{\prime}$) $\eta_{7}^{3}\simeq q\circ 4\circ[\nu_{7}]$. 

Consider the diagram 
\[\diagram 
     S^{8}\rto^-{[\nu_{5}\oplus\eta_{7}]}\dto^{\underline{4}} & SU(4)\dto^{4} \\ 
     S^{8}\rto^-{[\nu_{5}\oplus\eta_{7}]}\drto_{\nu_{5}\times\eta_{7}} 
         & SU(4)\dto^{q} \\ 
     & S^{5}\times S^{7}. 
  \enddiagram\] 
The top square homotopy commutes since the multiplications in 
$[S^{8},SU(4)]$ induced by the $H$-structure on $SU(4)$ and the 
co-$H$-structure on $S^{8}$ coincide. The bottom square homotopy 
commutes by definition of $[\nu_{5}\circ\eta_{7}]$. Since $\eta_{7}$ 
has order $2$ and, by Lemma~\ref{Todarelns1}~(c), $4\nu_{5}\simeq\eta_{5}^{3}$, 
we obtain $(\nu_{5}\times\eta_{7})\circ\underline{4}\simeq\eta_{5}^{3}$. 
Therefore $q\circ 4\circ [\nu_{5}\circ\eta_{7}]\simeq\eta_{5}^{3}$, and 
so (i$^{\prime}$) holds. 

Next, consider the diagram 
\[\diagram 
     S^{10}\rto^-{[\nu_{7}]}\dto^{\underline{4}} & SU(4)\dto^{4} \\ 
     S^{10}\rto^-{[\nu_{7}]}\drto_{\ast\times\nu_{7}} 
         & SU(4)\dto^{q} \\ 
     & S^{5}\times S^{7}. 
  \enddiagram\] 
The two squares homotopy commute as in the previous case. 
By Lemma~\ref{Todarelns1}~(c), $4\nu_{7}\simeq\eta_{7}^{3}$. Therefore   
$q\circ 4\circ\beta\simeq\eta_{7}^{3}$, and so (ii$^{\prime}$) holds. 
\end{proof}

\subsection{Dimension $15$}  

The relevant homotopy group from~\cite[Theorem 6.1]{MT} is: 
\begin{equation} 
  \label{SU4groups} 
  \begin{tabular}{|c|c|}\hline 
          & $\pi_{15}(SU(4))$ \\ \hline 
          $2$-component & $\mathbb{Z}/8\mathbb{Z}\oplus\mathbb{Z}/2\mathbb{Z}$ \\ \hline 
          generators & $[\nu_{5}\oplus\eta_{7}]\circ\sigma_{8}$, $[\sigma'\eta_{14}]$ \\ \hline 
  \end{tabular} 
\end{equation} 
In addition, Mimura and Toda~\cite[Lemma 6.2(i)]{MT} proved that 
\(\namedright{\pi_{16}(S^{5}\times S^{7})}{}{\pi_{15}(S^{3})}\) 
is an epimorphism, implying the following.  

\begin{lemma} 
   \label{pi15inj} 
   The map 
   \(\namedright{\pi_{15}(SU(4))}{q_{\ast}}{\pi_{15}(S^{5}\times S^{7})}\) 
   is an injection.~$\qqed$ 
\end{lemma} 

Next, we record information on $\pi_{15}(S^{7})$ determined by Toda~\cite[Theorem 7.6]{To}:  
\begin{equation} 
  \label{S15groups} 
  \begin{tabular}{|c|c|}\hline 
          & $\pi_{15}(S^{7})$ \\ \hline 
          $2$-component & $\mathbb{Z}/2\mathbb{Z}\oplus\mathbb{Z}/2\mathbb{Z} 
                  \oplus\mathbb{Z}/2\mathbb{Z}$  \\ \hline 
          generators & $\sigma'\eta_{14}$, $\bar{\nu}_{7}$, $\epsilon_{7}$ \\ \hline 
  \end{tabular} 
\end{equation} 
In addition, Toda~\cite{To} proved the following relations (the proofs are scattered 
through Toda's book but a summary list can be found in~\cite[Equations 1.1 and 2.1]{O}). 

\begin{lemma} 
   \label{Todarelns} 
   The following hold: 
   \begin{letterlist} 
      \item $\eta_{5}\bar{\nu}_{6}\simeq\nu_{5}^{3}$; 
      \item $\eta_{3}\nu_{4}\simeq\nu'_{3}\eta_{6}$; 
      \item $\eta_{6}\sigma'\simeq 4\bar{\nu}_{6}$; 
      \item $\eta_{6}\nu_{7}\simeq\nu_{6}\eta_{9}\simeq\ast$. 
   \end{letterlist} 
\end{lemma} 
\vspace{-1cm}~$\qqed$\medskip 

Lemma~\ref{Todarelns} is used to obtain two more relations. 

\begin{lemma} 
   \label{sphererelns} 
   The following hold: 
   \begin{letterlist} 
      \item $\eta_{5}^{2}\bar{\nu}_{7}\simeq\ast$; 
      \item $\eta_{5}^{2}\sigma'\simeq\ast$. 
   \end{letterlist} 
\end{lemma} 

\begin{proof} 
In what follows, we freely use the fact that the relations in Lemma~\ref{Todarelns} 
imply analogous relations for their suspensions; for example, 
$\eta_{5}\bar{\nu}_{6}\simeq\nu_{5}^{3}$ implies that $\eta_{6}\bar{\nu}_{7}\simeq\nu_{6}^{3}$. 

For part~(a), the relations in Lemma~\ref{Todarelns}~(a), (b) and (d) respectively imply 
the following string of equalities:  
$\eta_{5}^{2}\bar{\nu}_{7}\simeq\eta_{5}\nu_{6}^{3}\simeq\nu'_{5}\eta_{8}\nu_{9}^{2}\simeq\ast$.  

For part~(b), Lemma~\ref{Todarelns}~(c) and the fact that $\eta_{5}$ has order~$2$ 
imply that there are equalities $\eta_{5}^{2}\sigma'\simeq\eta_{5}(4\bar{\nu}_{6})\simeq\ast$. 
\end{proof} 

We now determine the homotopy classes of two maps into $SU(4)$. 

\begin{lemma} 
   \label{SU4relns} 
   The following hold: 
   \begin{letterlist} 
      \item the composite 
                \(\nameddright{S^{15}}{\bar{\nu}_{7}}{S^{7}}{d}{SU(4)}\) 
                is null homotopic; 
      \item the composite 
                \(\lnameddright{S^{15}}{\sigma'\eta_{14}}{S^{7}}{d}{SU(4)}\) 
                is null homotopic. 
   \end{letterlist} 
\end{lemma} 

\begin{proof} 
By Lemma~\ref{pi15inj}, 
\(\namedright{\pi_{15}(SU(4))}{q_{\ast}}{\pi_{15}(S^{5}\times S^{7})}\) 
is an injection. So in both cases it suffices to show that the assertions hold 
after composition with 
\(\namedright{SU(4)}{q}{S^{5}\times S^{7}}\). 
Since the composite 
\(\nameddright{S^{7}}{d}{SU(4)}{q}{S^{5}\times S^{7}}\) 
is $\eta_{5}^{2}\times\underline{2}$, the two assertions will follow if we prove: 
 
(a$^{\prime}$) $(\eta_{5}^{2}\times\underline{2})\circ\bar{\nu}_{7}\simeq\ast$; 
 
(b$^{\prime}$) $(\eta_{5}^{2}\times\underline{2})\circ\sigma'\eta_{14}\simeq\ast$. 

\noindent 
By Lemma~\ref{S72}, the degree two map on $S^{7}$ induces multiplication 
by $2$ on homotopy groups, so as both $\bar{\nu}_{7}$ 
and $\sigma'\eta_{14}$ have order~$2$, it suffices to prove: 

(a$^{\prime\prime}$) $\eta_{5}^{2}\bar{\nu}_{7}\simeq\ast$; 
 
(b$^{\prime\prime}$) $\eta_{5}^{2}\sigma'\eta_{14}\simeq\ast$. 

\noindent 
Part~(a$^{\prime\prime}$) is the statement of Lemma~\ref{sphererelns}~(a) and  
part~(b$^{\prime\prime}$) is immediate from Lemma~\ref{sphererelns}~(b). 
\end{proof} 

One consequence of Lemma~\ref{SU4relns} is the existence of an extension 
involving the space $E$ appearing in the homotopy decomposition of $\Sigma^{3} C$ 
in Lemma~\ref{Cdecomp}. 

\begin{lemma} 
   \label{Eext} 
   There is an extension 
   \[\diagram 
          S^{7}\rto^-{d}\dto^{\iota} & SU(4) \\ 
          E\urto_-{e} & 
     \enddiagram\] 
   for some map $e$. 
\end{lemma} 

\begin{proof} 
By Lemma~\ref{Cdecomp}, there is a homotopy cofibration 
\(\llnameddright{S^{18}}{u\cdot\bar{\nu}_{7}\nu_{15}}{S^{7}}{}{E}\) 
for some $u\in\mathbb{Z}/2\mathbb{Z}$. By Lemma~\ref{SU4relns}~(a), 
$d\circ\bar{\nu}_{7}$ is null homotopic. Therefore $d\circ (u\cdot\bar{\nu}_{7}\nu_{15})$ 
is null homotopic, implying that the asserted extension exists. 
\end{proof}

\section{The proof of Theorem~\ref{looppartialorder2}} 
\label{sec:proof} 

Recall from~(\ref{su4su2}) that 
\(\namedright{SU(4)}{\partial_{1}}{\Omega^{3}_{0} SU(4)}\) 
factors as the composite 
\(\nameddright{SU(4)}{q}{S^{5}\times S^{7}}{f}{\Omega^{3}_{0} SU(4)}\). 
Let 
\[f'\colon\namedright{\Sigma^{3}(S^{5}\times S^{7})}{}{SU(4)}\] 
be the triple adjoint of $f$. Let $f'_{1}$, $f'_{2}$ and $f'_{3}$ be the 
restrictions of the composite 
\[\nameddright{S^{8}\vee S^{10}\vee S^{15}}{\simeq}{\Sigma^{3}(S^{5}\vee S^{7})} 
     {f'}{SU(4)}\] 
to $S^{8}$, $S^{10}$ and $S^{15}$ respectively. We wish to identify 
$f'_{1}$, $f'_{2}$ and $f'_{3}$ more explicitly. Let 
\(t_{1}\colon\namedright{S^{5}}{}{SU(4)}\) 
and 
\(t_{2}\colon\namedright{S^{7}}{}{SU(4)}\) 
represent generators of $\pi_{5}(SU(4))\cong\mathbb{Z}$ and 
$\pi_{7}(SU(4))\cong\mathbb{Z}$ respectively. By~(\ref{SU4groups1}) these generators 
can be chosen so that $\pi\circ t_{1}$ is homotopic to $\underline{2}\oplus\ast$ 
and $\pi\circ t_{2}$ is homotopic to $\eta^{2}_{5}\oplus\underline{2}$. 
So there are homotopy commutative diagrams 
\begin{equation} 
  \label{2Bottdgrms} 
  \diagram 
     S^{5}\rto^-{t_{1}}\drto_{\underline{2}\oplus\ast} 
        & SU(4)\rto^-{\partial_{1}}\dto^{q} & \Omega^{3}_{0} SU(4)\ddouble 
        & & S^{7}\rto^-{t_{2}}\drto_{\eta^{2}_{5}\oplus\underline{2}} 
        & SU(4)\rto^-{\partial_{1}}\dto^{q} & \Omega^{3}_{0} SU(4)\ddouble \\ 
     & S^{5}\times S^{7}\rto^-{f} & \Omega^{3}_{0} SU(4) 
        & & & S^{5}\times S^{7}\rto^-{f} & \Omega^{3}_{0} SU(4). 
  \enddiagram 
\end{equation}  
On the other hand, since the triple adjoint of $\partial_{1}$ is the Samelson 
product $\langle i,1\rangle$, the triple adjoint of $\partial_{1}\circ t_{j}$ 
is $\langle t_{j},1\rangle$ for $j=1,2$. Bott~\cite[Theorem 1]{B} calculated that both of 
these maps have order~$4$. Thus the left diagram in~(\ref{2Bottdgrms}) 
implies that the restriction of $f$ to $S^{5}$ has order~$8$, and the right 
diagram in~(\ref{2Bottdgrms}) implies that the restriction of $f$ to $S^{7}$ 
has order~$8$. Thus, taking triple adjoints, $f'_{1}$ and $f'_{2}$ both 
have order~$8$. 

The order of $f'_{3}$ is not as clear. By~(\ref{SU4groups}), 
$\pi_{15}(SU(4))\cong\mathbb{Z}/8\mathbb{Z}\oplus\mathbb{Z}/2\mathbb{Z}$, 
so $f'_{3}$ may have order~$8$. This ambiguity will be reflected in the alternative 
possibilities worked out below.  

Recall from Lemma~\ref{jclass} that there is a homotopy commutative diagram 
\[\diagram 
       S^{8}\vee S^{10}\vee S^{15}\rto^-{\simeq}\dto^{a+b+c}  
            & \Sigma^{3}(S^{5}\times S^{7})\dto^{\Sigma^{3} j} \\ 
       E\vee S^{12}\vee S^{14}\rto^-{\simeq} & \Sigma^{3} C   
  \enddiagram\]   
where $a$, $b$ and $c$ respectively are the composites 
\begin{align*} 
       & a\colon\nameddright{S^{8}}{\eta_{7}}{S^{7}}{\iota}{E}\hookrightarrow 
             E\vee S^{12}\vee S^{14} \\  
       & b\colon\nameddright{S^{10}}{\nu'_{7}}{S^{7}}{\iota}{E}\hookrightarrow 
             E\vee S^{12}\vee S^{14} \\ 
       & c\colon\llllnameddright{S^{15}}{\psi+s\cdot\nu'_{12}+\eta_{14}}{S^{7}\vee S^{12}\vee S^{14}} 
             {\iota\vee 1\vee 1}{E\vee S^{12}\vee S^{14}} 
\end{align*} 
and $\psi=t\cdot\sigma'\eta_{14}$ for some $t\in\mathbb{Z}/2\mathbb{Z}$. 
Let $c'$ be the composite 
\[c'\colon\llllnameddright{S^{15}}{\psi'+s\cdot\nu'_{12}+\eta_{14}}{S^{7}\vee S^{12}\vee S^{14}} 
     {\iota\vee 1\vee 1}{E\vee S^{12}\vee S^{14}}\] 
where $\psi'=t\cdot\sigma'\eta_{14}+\eta_{7}\sigma_{8}$. Let $\xi$ 
be the composite 
\[\xi\colon\nameddright{E\vee S^{12}\vee S^{14}}{}{E}{e}{SU(4)}\] 
where the left map is the pinch onto the first wedge summand and $e$ is 
the map from Lemma~\ref{Eext}. 

\begin{lemma} 
   \label{alternative} 
   There is a homotopy commutative diagram 
   \[\diagram 
          S^{8}\vee S^{10}\vee S^{15}\rrto^-{f'_{1}+f'_{2}+f'_{3}}\dto^{a+b+\gamma} 
             & & SU(4)\dto^{4} \\ 
          E\vee S^{12}\vee S^{14}\rrto^-{\xi} & & SU(4) 
     \enddiagram\] 
   where $\gamma$ may be chosen to be $c$ if the order of $f'_{3}$ is at most $4$ 
   and $\gamma$ may be chosen to be $c'$ if the order of $f'_{3}$ is $8$. Further, 
   in the latter case, the composite 
   \(\namedddright{S^{15}}{\eta_{7}\sigma_{8}}{S^{7}}{\iota}{E}{e}{SU(4)}\) 
   represents $4[\nu_{5}\oplus\eta_{7}]\circ\sigma_{8}$.
\end{lemma} 

\begin{proof} 
First, consider the diagram 
\begin{equation} 
  \label{dgrm1} 
  \diagram
       S^{8}\vee S^{10}\rto^-{f'_{1}+f'_{2}}\dto^{\eta_{7}+\nu'_{7}} 
            & SU(4)\dto^{4} \\ 
       S^{7}\rto^-{d}\dto^{\iota} & SU(4)\ddouble \\ 
       E\rto^-{e} & SU(4).  
  \enddiagram 
\end{equation}  
Since $\pi_{8}(SU(4))\cong\mathbb{Z}/8\mathbb{Z}$ is generated by 
$[\nu_{5}\oplus\eta_{7}]$ and $f'_{1}$ has order~$8$, we must have 
$f'_{1}=u\cdot[\nu_{5}\oplus\eta_{7}]$ for some unit $u\in\mathbb{Z}/8\mathbb{Z}$. 
Thus $4f'_{1}\simeq 4[\nu_{5}\oplus\eta_{7}]$, so the restriction of the upper square 
in~(\ref{dgrm1}) to~$S^{8}$ homotopy commutes by Lemma~\ref{S810dgrms}. Similarly,  
since $\pi_{10}(SU(4))\cong\mathbb{Z}/8\mathbb{Z}\oplus\mathbb{Z}/2\mathbb{Z}$ 
with $[\nu_{7}]$ being the generator of order~$8$, and $f'_{2}$ has order $8$, we 
must have $4f'_{2}\simeq 4[\nu_{7}]$, so the restriction of the upper square 
in~(\ref{dgrm1}) to $S^{10}$ homotopy commutes by Lemma~\ref{S810dgrms}. 
The lower square in~(\ref{dgrm1}) homotopy commutes by Lemma~\ref{Eext}. 
Now observe that the anticlockwise path around~(\ref{dgrm1}) is the definition of $\xi\circ(a+b)$. 
Thus~(\ref{dgrm1}) implies that the diagram in the statement of the lemma 
homotopy commutes when restricted to $S^{8}\vee S^{10}$. 

Second, consider the diagram 
\begin{equation} 
  \label{dgrm2} 
  \diagram 
      S^{15}\dto^{(t\cdot\sigma'\eta_{14}+\theta)+s\cdot\nu'_{12}+\eta_{14}}\rrto^-{f'_{3}} 
           & & SU(4)\ddto^{4} \\ 
      S^{7}\vee S^{12}\vee S^{15}\dto^{\iota\vee 1\vee 1} & & \\ 
      E\vee S^{12}\vee S^{14}\rrto^-{\xi} & & SU(4) 
  \enddiagram 
\end{equation}   
where two possibilities for $\theta$ will be considered. At the bottom of the diagram, by definition, $\xi$ is the composite 
\(\nameddright{E\vee S^{12}\vee S^{14}}{}{E}{e}{SU(4)}\) 
where the left map is the pinch onto the first wedge summand. By 
Lemma~\ref{Eext}, $e\circ\iota=d$. Thus the anticlockwise way around the 
diagram is homotopic to the composite 
\(\lllnameddright{S^{15}}{t\cdot\sigma'\eta_{14}+\theta}{S^{7}}{d}{SU(4)}\).  
By Lemma~\ref{SU4relns}~(b), $d\circ t\cdot\sigma'\eta_{14}$ is null homotopic. 
Thus the lower direction around the diagram is in fact homotopic to the composite 
\(\nameddright{S^{15}}{\theta}{S^{7}}{d}{SU(4)}\). 

If $f'_{3}$ has order at most $4$ then $4f'_{3}$ is null homotopic. Taking  
$\theta$ to be the constant map shows that~(\ref{dgrm2}) homotopy 
commutes. Observe also that with this choice of $\theta$ the left column 
in~(\ref{dgrm2}) is the definition of $c$, so we obtain the diagram in the 
statement of the lemma when restricted to $S^{15}$. Now combining~(\ref{dgrm1}) 
and~(\ref{dgrm2}) we obtain the diagram asserted by the lemma. 

Suppose that $f'_{3}$ has order~$8$. Since 
$\pi_{15}(SU(4))\cong\mathbb{Z}/8\mathbb{Z}\oplus\mathbb{Z}/2\mathbb{Z}$ 
with the order~$8$ generator being $[\nu_{5}\oplus\eta_{7}]\circ\sigma_{8}$, 
we obtain $4f'_{3}\simeq 4[\nu_{5}\oplus\eta_{7}]\circ\sigma_{8}$. Take  
$\theta=\eta_{7}\sigma_{8}$. We claim that 
$d\circ\theta\simeq 4[\nu_{5}\circ\eta_{7}]\circ\sigma_{8}$. If so then~(\ref{dgrm2}) 
homotopy commutes with this choice of $\theta$ and, as the left column of~(\ref{dgrm2}) 
is the definition of $c'$, we obtain the diagram in the statement of the lemma 
when restricted to $S^{15}$. Therefore combining~(\ref{dgrm1}) and~(\ref{dgrm2}) 
we obtain the diagram asserted by the lemma.  

It remains to show that 
$d\circ\eta_{7}\sigma_{8}\simeq 4[\nu_{5}\oplus\eta_{7}]\circ\sigma_{8}$. 
By Lemma~\ref{pi15inj} it suffices to compose with 
\(\namedright{SU(4)}{q}{S^{5}\times S^{7}}\) 
and check there. On the one hand, 
$q\circ d\circ\eta_{7}\sigma_{8}\simeq 
    (\eta_{5}^{2}\times\underline{2})\circ\eta_{7}\sigma_{8}\simeq 
    \eta_{5}^{3}\sigma_{8}$,  
where the left homotopy holds by definition of $d$ and the right homotopy is 
due to the fact that $\eta_{7}$ has order~$2$ and, by Lemma~\ref{S72}, 
$\underline{2}$ induces multiplication by $2$ on homotopy groups. On 
the other hand, 
$q\circ 4[\nu_{5}\oplus\eta_{7}]\circ\sigma_{8}\simeq 
     4(\nu_{5}\times\eta_{7})\circ\sigma_{8}\simeq 4\nu_{5}\sigma_{8}\simeq 
     \eta_{5}^{3}\sigma_{8}$. 
Here, from left to right, the first homotopy holds by definition of $[\nu_{5}\oplus\eta_{7}]$, 
the second holds since $\eta_{7}$ has order~$2$, and the third holds by Lemma~\ref{Todarelns1}~(b). 
Thus $d\circ\eta_{7}\circ\sigma_{8}\simeq 4[\nu_{5}\oplus\eta_{7}]\circ\sigma_{8}$, 
as claimed. 
\end{proof} 

Now return to the map 
\(\namedright{SU(4)}{\partial_{1}}{\Omega_{0}^{3} SU(4)}\). 

\begin{proposition} 
   \label{order4options} 
   The following hold: 
   \begin{letterlist} 
      \item if $f'_{3}$ has order at most $4$ then $4\circ\partial_{1}$ is null homotopic;  
      \item if $f'_{3}$ has order $8$ then $4\circ\partial_{1}$ is homotopic to the composite 
                \(\namedddright{SU(4)}{q}{S^{5}\times S^{7}}{}{S^{12}}{4\chi} 
                      {\Omega^{3}_{0} SU(4)}\), 
                 where the middle map is the pinch map to the top cell and $\chi$ 
                 is the triple adjoint of the order~$8$ generator 
                 $[\nu_{5}\oplus\eta_{7}]\circ\sigma_{8}$ in $\pi_{15}(SU(4))$. 
   \end{letterlist} 
\end{proposition} 

\begin{proof} 
If the order of $f'_{3}$ is at most~$4$, then in Lemma~\ref{alternative} we 
may take $\gamma=c$. Doing so, observe that by using the inverse 
equivalences in Lemma~\ref{jclass} we obtain a homotopy commutative diagram 
\begin{equation} 
  \label{Cjdgrm} 
  \diagram 
        \Sigma^{3}(S^{5}\times S^{7})\rto^-{f'}\dto^{\Sigma^{3} j} & SU(4)\dto^{4} \\ 
        \Sigma^{3} C\rto^-{\xi'} & SU(4) 
  \enddiagram 
\end{equation}  
where $\xi'$ is the composite 
\(\nameddright{\Sigma^{3} C}{\simeq}{E\vee S^{12}\vee S^{14}}{\xi}{SU(4)}\). 
Now consider the diagram 
\[\diagram 
      SU(4)\rto^-{\partial_{1}}\dto^{q} & \Omega^{3}_{0} SU(4)\ddouble \\ 
      S^{5}\times S^{7}\rto^-{f}\dto^{j} & \Omega^{3}_{0} SU(4)\dto^{4} \\ 
      C\rto & \Omega^{3}_{0} SU(4) 
  \enddiagram\]
The top square homotopy commutes by~(\ref{su4su2}) while the bottom 
square is the triple adjoint of~(\ref{Cjdgrm}). Since the left 
column consists of two consecutive maps in a homotopy cofibration sequence  
it is null homotopic. The homotopy commutativity of the diagram therefore 
implies that~$4\circ\partial_{1}$ is null homotopic. 

If the order of $f'_{3}$ is $8$, then in Lemma~\ref{alternative} we may 
take $\gamma=c'$. Doing so, since $c'=c+\eta_{7}\sigma_{8}$, instead 
of~(\ref{Cjdgrm}) we obtain a homotopy commutative diagram 
\begin{equation} 
  \label{Cj2dgrm} 
  \diagram 
        \Sigma^{3}(S^{5}\times S^{7})\rto^-{f'}\dto^{\Sigma^{3} j+\ell} & SU(4)\dto^{4} \\ 
        \Sigma^{3} C\rto^-{\xi'} & SU(4) 
  \enddiagram 
\end{equation}  
where $\ell$ is the composite 
\(\nameddright{\Sigma^{3}(S^{5}\times S^{7})}{}{S^{15}}{\eta_{7}\sigma_{8}} 
     {S^{7}}\hookrightarrow\Sigma^{3} C\). 
Now consider the diagram 
\[\diagram 
      \Sigma^{3} SU(4)\rto^-{\partial'_{1}}\dto^{\Sigma^{3} q} & SU(4)\ddouble \\ 
      \Sigma^{3}(S^{5}\times S^{7})\rto^-{f'}\dto^{\Sigma^{3} j+\ell} & SU(4)\dto^{4} \\ 
      \Sigma^{3} C\rto^-{\xi'} & SU(4) 
  \enddiagram\] 
where $\partial'_{1}$ is the triple adjoint of $\partial$. The top square homotopy 
commutes by~(\ref{su4su2}) while the bottom square homotopy commutes 
by~(\ref{Cj2dgrm}). Since $\Sigma^{3} j\circ\Sigma^{3} q$ are consecutive 
maps in a homotopy cofibration, their composite is null homotopic. Thus 
this diagram implies that $4\circ\partial'_{1}$ is homotopic to the composite 
\(\namedddright{\Sigma^{3} SU(4)}{\Sigma^{3} q}{\Sigma^{3}(S^{5}\times S^{7})} 
     {}{S^{15}}{\eta_{7}\sigma_{8}}{S^{7}}\hookrightarrow\namedright{\Sigma^{3} C} 
     {\xi'}{SU(4)}\). 
Notice that the pinch map to the top cell 
\(\namedright{\Sigma^{3}(S^{5}\times S^{7})}{}{S^{15}}\) 
is a triple suspension, while by Lemma~\ref{alternative} the composite 
\(\namedright{S^{15}}{\eta_{7}\sigma_{8}}{S^{7}}\hookrightarrow\namedright{\Sigma^{3} C} 
        {\xi'}{SU(4)}\) 
represents $4[\nu_{5}\oplus\eta_{7}]\circ\sigma_{8}$. Thus, taking triple adjoints, 
$4\circ\partial_{1}$ is homotopic to the composite 
\(\namedddright{SU(4)}{q}{S^{5}\times S^{7}}{}{S^{12}}{4\chi}{SU(4)}\), 
as asserted.  
\end{proof} 

\begin{remark} 
It can be checked that if $f'_{3}$ has order~$8$ then there does not exist a 
map~$\xi$ such that $\xi\circ(a+b+c)\simeq 4f'$ in Lemma~\ref{alternative}. 
The argument is to check all possible cases; it is not included as it is not needed. 
This leads to the conclusion that if $f'_{8}$ has order~$8$ then 
$4\circ\partial_{1}$ is nontrivial; for if it were trivial then 
$4\circ\partial_{1}\simeq 4\circ f\circ\pi$ 
would have to factor through the cofibre $C$ of $\pi$, implying that there 
is a map $\xi$ such that $\xi\circ(a+b+c)\simeq 4f'$.
\end{remark}

\begin{theorem} 
   \label{partialorder4} 
   The following hold: 
   \begin{letterlist} 
      \item if $f'_{3}$ has order $4$ then $\partial_{1}$ has order $4$; 
      \item if $f'_{3}$ has order $8$ then $\Omega\partial_{1}$ has order $4$. 
   \end{letterlist} 
\end{theorem} 

\begin{proof} 
By Proposition~\ref{order4options}, if $f'_{3}$ has order $4$ then $4\circ\partial_{1}$ 
is null homotopic, implying that $\partial_{1}$ has order at most $4$. On the 
other hand, by Lemma~\ref{HKlemma}, the order of $\partial_{1}$ is divisible by $4$. 
Thus $\partial_{1}$ has order $4$. 

Next, in general, the quotient map 
\(\namedright{X\times Y}{Q}{X\wedge Y}\) 
is null homotopic after looping. For if 
\(i\colon\namedright{X\vee Y}{}{X\times Y}\) 
is the inclusion of the wedge into the product then $\Omega i$ has a right 
homotopy inverse. Therefore $\Omega Q$ factors through $\Omega Q\circ\Omega i$, 
which is null homotopic since $Q\circ i$ is.  
In our case, if $f'_{3}$ has order $8$ then 
Proposition~\ref{order4options} states that $4\circ\partial_{1}$ 
factors through the quotient map 
\(\namedright{S^{5}\times S^{7}}{Q}{S^{5}\wedge S^{7}\simeq S^{12}}\). 
Thus $4\Omega\partial_{1}$ is null homotopic. Consequently, $\Omega\partial_{1}$ 
has order at most $4$. By Lemma~\ref{loopHKlemma}, the order  
of~$\Omega\partial_{1}$ is divisible by $4$. Thus $\Omega\partial_{1}$ has order $4$. 
\end{proof} 

\begin{proof}[Proof of Theorem~\ref{looppartialorder2}] 
Theorem~\ref{partialorder4} implies that in any case the 
$2$-primary component of the order of $\Omega\partial_{1}$ is $4$. 
\end{proof}

\bibliographystyle{amsplain}

\begin{thebibliography}{10}

\bibitem{AB} M.F. Atiyah and R. Bott, The Yang-Mills equations over
   Riemann surfaces, \emph{Philos. Trans. Roy. Soc. London Ser. A}
   \textbf{308} (1983), 523-615.
\bibitem{BH} J. Baez, J. Huerta, The algebra of grand unified theories, \emph{Bull. Am. Math. Soc.} \textbf{47} (2010), 483-552.
\bibitem{B} R. Bott, A note on the Samelson product in the 
   classical Lie groups, \emph{Comment. Math. Helv.} \textbf{34} 
   (1960), 245-256. 
\bibitem{Cr} M.C. Crabb, On the stable splitting of $U(n)$ and $\Omega U(n)$, 
   \emph{Algebraic Topology, Barcelona, 1986}, 35-53, Lecture Notes in Math. 
   \textbf{1298}, Springer, Berlin, 1987. 
\bibitem{CS} M.C. Crabb and W.A. Sutherland, Counting homotopy
   types of gauge groups, \emph{Proc. London Math. Soc.} \textbf{83}
   (2000), 747-768.    
\bibitem{Cu} T. Cutler, Homotopy types of $Sp(3)$-gauge groups, \emph{Topology Appl.} 
   \textbf{236} (2018), 44-58. 
\bibitem{G} D.H. Gottlieb, Applications of bundle map theory, 
   \emph{Trans. Amer. Math. Soc.} \textbf{171} (1972), 23-50.
\bibitem{tH} G. `t Hooft, A planar diagram theory for strong interactions, \emph{Nucl. Phys. B} \textbf{72} (1974), 461-473.
\bibitem{HK} H. Hamanaka and A. Kono, Unstable $K^{1}$-group and
   homotopy type of certain gauge groups, \emph{Proc. Roy. Soc. Edinburgh
   Sect. A} \textbf{136} (2006), 149-155. 
\bibitem{JW} I.M. James and J.H.C. Whitehead, The homotopy theory of sphere 
   bundles over spheres. I, \emph{Proc. London Math. Soc.} \textbf{4} (1954), 196-218. 
\bibitem{Ka} R. Kane, The homology of Hopf spaces, \emph{North-Holland}, Amsterdam (1988).
\bibitem{KTT} D. Kishimoto, S. Theriault and M. Tsutaya, Homotopy types 
   of $G_{2}$-gauge groups, \emph{Topology Appl.} \textbf{228} (2017), 92-107. 
\bibitem{Kitch} N. Kitchloo, Cohomology splittings of Stiefel manifolds, 
   \emph{J. London Math. Soc.} \textbf{64} (2001), 457-471. 
\bibitem{K} A. Kono, A note on the homotopy type of certain gauge
   groups, \emph{Proc. Roy. Soc. Edinburgh Sect. A} \textbf{117} (1991),
   295-297. 
\bibitem{L} G.E. Lang, The evaluation map and $EHP$ sequences,
   \emph{Pacific J. Math.} \textbf{44} (1973), 201-210. 
\bibitem{M} H. Miller, Stable splittings of Stiefel manifolds, \emph{Topology} 
   \textbf{24} (1985), 411-419. 
\bibitem{MT} M. Mimura, and H. Toda, Homotopy groups of $SU(3)$,
   $SU(4)$, and $Sp(2)$, \emph{J. Math. Kyoto Univ.} \textbf{3} (1964), 217-250.  
\bibitem{O} K. Oguchi, Generators of $2$-primary components of homotopy 
   groups of spheres, unitary groups and symplectic groups, \emph{J. Fac. Sci. 
   Univ. Tokyo} \textbf{11} (1964), 65-111.
\bibitem{PS} J. Pati and A. Salam, Lepton number as the fourth color, 
   \emph{Phys. Rev. D} \textbf{10} (1974), 275-289.
\bibitem{S} W.A. Sutherland, Function spaces related to gauge groups,
   \emph{Proc. Roy. Soc. Edinburgh Sect. A} \textbf{121} (1992), 185-190. 
\bibitem{Th1} S.D. Theriault, The homotopy types of $Sp(2)$-gauge 
   groups, \emph{Kyoto J. Math.} \textbf{50} (2010), 591-605. 
\bibitem{Th2} S.D. Theriault, The homotopy types of $SU(5)$-gauge groups, 
   \emph{Osaka J. Math.} \textbf{52} (2015), 15-29. 
\bibitem{Th3} S.D. Theriault, Odd primary homotopy types of $SU(n)$-gauge groups,  
   \emph{Algebr. Geom. Topol.} \textbf{17} (2017), 1131-1150. 
\bibitem{To} H. Toda, \emph{Composition methods in homotopy groups 
   of spheres}, Annals of Mathematics Studies \textbf{49}, Princeton University 
   Press, Princeton N.J., 1962. 

\end{thebibliography}

\end{document}